\documentclass[11pt,reqno]{amsart}

\usepackage[hmargin=4cm,vmargin=4cm]{geometry}

\usepackage[leqno]{amsmath}

\usepackage[english]{babel}

\usepackage{tikz}
\usepackage{tikz-cd}
\usepackage{faktor}

\newtheorem  {theorem}                  {Theorem}
\newtheorem* {theorem*}                   {Theorem}

\newtheorem {lemma}[theorem] {Lemma}
\newtheorem {prop}[theorem]      {Proposition}
\newtheorem* {prop*}     {Proposition}

\newtheorem {Corollary}[theorem]                 {Corollary}
\newtheorem {corollary}[theorem]      {Corollary}

\theoremstyle{definition}
\newtheorem {defi}[theorem] {Definition}
\newtheorem {Remark} [theorem]         {Remark}

\newtheorem* {Example*}    {Example}

\newcommand{\T}{\mathbb{T}}

\newcommand{\curl}{\operatorname{curl}}

\newcommand{\dive}{\operatorname{div}}

\newlength{\intwidth}

\DeclareRobustCommand{\Bint}
   {\mathop{%
      \text{%
        \settowidth{\intwidth}{$\int$}%
        \makebox[0pt][l]{\makebox[\intwidth]{$-$}}%
        $\int$}}}

\newcommand{\bE}{\mathbb{E}}

\renewenvironment{abstract}[1]
  {\medskip\selectlanguage{#1}%
   \textsc{\abstractname.}}
  {\par\bigskip}

\usepackage{lipsum}

\title{Computability and Beltrami fields in Euclidean~space}

\author{Robert Cardona}\address{ Robert Cardona,
Laboratory of Geometry and Dynamical Systems,  Universitat Polit\`{e}cnica de Catalunya,  Avinguda del Doctor Mara\~{n}on 44-50, 08028, Barcelona   \it{e-mail: robert.cardona@upc.edu}}
 \thanks{Robert Cardona acknowledges financial support from the Spanish Ministry of Economy and Competitiveness, through the Mar\'ia de Maeztu Programme for Units of Excellence in R\& D (MDM-2014-0445) via an FPI grant.}
\author{Eva Miranda}\address{ Eva Miranda,
Laboratory of Geometry and Dynamical Systems $\&$ Institut de Matem\`atiques de la UPC-BarcelonaTech (IMTech),  Universitat Polit\`{e}cnica de Catalunya,  Avinguda del Doctor Mara\~{n}on 44-50, 08028, Barcelona  \\  CRM Centre de Recerca Matem\`{a}tica, Campus de Bellaterra
Edifici C, 08193 Bellaterra, Barcelona
 \it{e-mail: eva.miranda@upc.edu }
 }
\thanks{Robert Cardona and Eva Miranda are partially supported by the AEI grant PID2019-103849GB-I00 of MCIN/ AEI /10.13039/501100011033. Eva Miranda is supported by the Catalan Institution for Research and Advanced Studies via an ICREA Academia Prize 2016 and ICREA Academia Prize 2021 and by the Spanish State
Research Agency, through the Severo Ochoa and Mar\'{\i}a de Maeztu Program for Centers and Units
of Excellence in R\&D (project CEX2020-001084-M)}
\author{Daniel Peralta-Salas} \address{Daniel Peralta-Salas, Instituto de Ciencias Matem\'aticas, Consejo Superior de Investigaciones Cient\'ificas,
28049 Madrid, Spain. \it{e-mail: dperalta@icmat.es} }
\thanks{Daniel Peralta-Salas is supported by the grants CEX2019-000904-S, RED2018-102650-T, EUR2019-103821 and PID2019-106715GB GB-C21 funded by MCIN/AEI/ 10.13039/501100011033. The authors were supported by the project Computational, dynamical and geometrical complexity in fluid dynamics -  AYUDAS FUNDACIÓN BBVA A PROYECTOS INVESTIGACIÓN CIENTÍFICA 2021.}

\begin{document}

\maketitle

\begin{abstract}{english}
In this article, we pursue our investigation of the connections between the theory of computation and hydrodynamics. We prove the existence of stationary solutions of the Euler equations in Euclidean space, of Beltrami type, that can simulate a universal Turing machine. In particular, these solutions possess undecidable trajectories. Heretofore, the known Turing complete constructions of steady Euler flows in dimension 3 or higher were not associated to a prescribed metric. Our solutions do not have finite energy, and their construction makes crucial use of the non-compactness of $\mathbb R^3$, however they can be employed to show that an arbitrary tape-bounded Turing machine can be robustly simulated by a Beltrami flow on $\mathbb T^3$ (with the standard flat metric). This shows that there exist steady solutions to the Euler equations on the flat torus exhibiting dynamical phenomena of (robust) computational complexity as high as desired. We also quantify the energetic cost for a Beltrami field on $\mathbb T^3$ to simulate a tape-bounded Turing machine, thus providing additional support for the space-bounded Church-Turing thesis. Another implication of our construction is that a Gaussian random Beltrami field on Euclidean space exhibits arbitrarily high computational complexity with probability~$1$. Finally, our proof also yields Turing complete flows and diffeomorphisms on $\mathbb{S}^2$ with zero topological entropy, thus disclosing a certain degree of independence within different hierarchies of complexity.
\end{abstract}

\section{Introduction}\label{S:intro}

A vector field $u$ in $\mathbb R^3$ that satisfies the equation
\[
\curl u=\lambda u
\]
for some constant $\lambda\neq 0$ is called a \emph{Beltrami field}. These fields appear, and are extremely relevant, in many physical contexts. In fluid mechanics, they are stationary solutions to the \emph{Euler equations} with constant Bernoulli function; in magnetohydrodynamics, they are known as \emph{force-free fields}, and model stellar atmospheres and plasma equilibria.

The computability/hydrodynamics duet has attracted considerable attention in recent years motivated both by the foundational works of Moore in the 1990s, and Tao's programme on the universality of the Euler equations. More generally, it is relevant to understand the computational complexity of a given class of dynamical systems, and the decidability of properties such as the reachability problem of an orbit~\cite{Mo90, Mo1} or the computability of dynamically-defined sets~\cite{BY}. Tao's \emph{universality} programme~\cite{T1,T2,T3} speculates on the connections between Turing complete solutions in hydrodynamics and the blow-up problem for the Euler and the Navier-Stokes equations~\cite{TNat}. If the Riemannian metric is not fixed (and hence it can be considered as an additional ``free variable'' of the problem), geometric techniques from the symplectic and contact worlds turn out to be extremely useful, as exploited in~\cite{CMPP1,CMPP2} to construct Turing complete steady Euler flows on some Riemannian spheres. Such constructions address a question raised by  Moore in 1991 of whether a fluid flow is capable of performing computations~\cite{Mo90,Mo1}, that we tackle in this paper for the Euclidean three-dimensional space.

In the constructions of stationary solutions of the Euler equations in~\cite{CMPP1, CMPP2} the Riemannian metric is deformed, adapted to the (underlying) contact geometry. Some of these deformations are controlled using sophisticated techniques such as the $h$-principle, or realizing Turing complete diffeomorphisms of the disk as Poincar\'{e} return maps of Reeb flows. This can make the metric very intricate in certain regions. Time-dependent solutions of the Euler equations exhibiting different complex behaviors are also constructed in~\cite{T2, T3, TdL, CMP2}, but again they make use of non-canonical metrics. In particular, the solutions of the Euler equations obtained in~\cite{CMPP1,CMPP2,CMP2} {model Turing complete fluid flows for non-Euclidean metrics}. Understanding the Riemannian properties of these metrics can be challenging, and these techniques fail when trying to obtain Turing complete solutions in Euclidean space.

In this work, we fill this gap establishing the existence of a Turing complete Beltrami field in $\mathbb R^3$. Accordingly, we obtain a universal Turing machine simulated by a $3$-dimensional steady fluid flow in Euclidean space, {which is the usual context of mathematical hydrodynamics}. The techniques of proof are completely different from those applied in~\cite{CMPP1,CMPP2}. Specifically, this is our main result:

\begin{theorem}\label{th.main}
There exists a Beltrami field $u$ in $\mathbb R^3$ that can simulate a universal Turing machine (i.e., the field is Turing complete).
\end{theorem}

A comparison between Theorem~\ref{th.main} and the main result in~\cite{CMPP2} is in order. As mentioned earlier, the field $\widetilde u$ constructed in~\cite{CMPP2} is a Beltrami field on $\mathbb R^3$ for some Riemannian metric, which is not the canonical one. In contrast, the vector field $u$ in Theorem~\ref{th.main} is Beltrami with respect to the Euclidean metric. However, while the universal Turing machine is simulated by the flow of $\widetilde u$ restricted to a compact invariant set of $\mathbb R^3$, all the Turing machine computations of $u$ occur on an invariant set that is non-compact (and contained in the plane $\{z=0\}$). Accordingly, the non-compactness of $\mathbb R^3$ is crucially used in the proof of Theorem~\ref{th.main}. In particular, $\widetilde u$ has finite energy ($L^2$-norm), which is not the case of $u$.


The idea to prove Theorem~\ref{th.main} consists in constructing an analytic vector field $X$ tangent to the plane $\{z=0\}\subset\mathbb R^3$ that is Turing complete and then extending it to a Beltrami field $u$ on $\mathbb R^3$ so that $u|_{\{z=0\}}=X$. The typical extension tool in the context of analytic PDEs is the \emph{Cauchy-Kovalevskaya} theorem. Two difficulties arise in the implementation of this strategy. First, the only available Cauchy-Kovalevskaya theorem for Beltrami fields gives solutions defined just on a neighborhood of the Cauchy surface, cf.~\cite[Theorem 3.1]{EP12}; second, since the curl operator is not elliptic, a necessary condition to apply Cauchy-Kovalevskaya with a Cauchy datum $X$ on $\{z=0\}$ is that it must be a gradient of an analytic function, so we need to construct a Turing complete system of the form $\nabla_{\mathbb R^2} F$.

We overcome these problems by proving a new Cauchy-Kovalevskaya theorem for the curl operator that gives global Beltrami fields if the planar function $F$ is \emph{entire}, in the sense that it admits a holomorphic extension to $\mathbb C^2$. To ensure that $\nabla_{\mathbb R^2} F$ is Turing complete, we construct a smooth function $\widehat F$ whose gradient simulates a universal Turing machine in a robust enough sense so that an entire approximation $F$ of $\widehat F$ is still Turing complete. The construction of $\widehat F$ is the most technically demanding part of this work and requires a new method to encode Turing machine dynamics into a planar flow. Our encoding and idea of fixing certain orbits of the system is inspired by the Turing complete countably piecewise linear function of the interval constructed in~\cite{KCG} by Koiran, Cosnard and Garzon. The flow that we construct is a dynamical system which, like theirs, is computable but it does not have an ``explicit" finite description, in contrast with other examples of Turing complete systems~\cite{Mo1, GCB, GCB2, GZ}.

Our construction, combined with the inverse localization theorem established in~\cite{EPT}, allows us to prove that Beltrami fields on the flat torus $\mathbb T^3$ are capable of robustly simulating space-bounded computations {(that is, Turing machines with a finite number of configurations)}. These simulations can exhibit computational complexity as high as desired if the $H^1$-norm of the solution is large enough. {By computational complexity here we mean the amount of resources (i.e., the number of steps and the tape size) needed to perform a computation, either in time or space.} This energetic cost prevents us from proving the existence of Turing complete steady Euler flows on $\mathbb T^3$. This limitation of the {robust} computational power in terms of the energy (or memory) of the physical system is fully consistent with the space-bounded Church-Turing thesis~\cite{BSR}.

\begin{theorem}\label{thm:torus}
Any tape-bounded Turing machine can be robustly simulated by a Beltrami field $v$ on the flat torus $\T^3$. From a quantitative viewpoint, if $s_b$ is the tape size of the tape-bounded Turing machine, then the $H^1$-norm of $v$ is at least
\begin{equation}\label{Eq.LB}
\|v\|_{H^1(\mathbb T^3)}\geq C\exp(\exp(\exp(Cs_b)))
\end{equation}
for some positive constant $C$.
\end{theorem}

\begin{Remark}
The lower bound~\eqref{Eq.LB} of the $H^1$-norm of $v$ in terms of the tape size of the Turing machine that it simulates comes from the specific way we encode the Turing machine into the dynamics of $v$ (i.e., the construction of the Turing complete planar system $\nabla_{\mathbb R^2} F$ and the inverse localization theorem). In particular, we do not claim that all implementations of tape-bounded Turing machines using steady Euler flows on $\mathbb T^3$ necessarily obey such a bound.
\end{Remark}

In this theorem, the notion of ``simulation" is a bit different from the one in Theorem~\ref{th.main}, as explained in Section~\ref{S:bounded}. In particular, the robustness of the simulations holds with respect to uniform perturbations of the flow and of the initial conditions. Theorem~\ref{thm:torus} implies that there are Beltrami field on $\mathbb T^3$ whose orbits exhibit computational complexity as high as desired. This computational complexity arises in the reachability problem of determining whether the orbit of the field through an explicit point will intersect an explicit open set before leaving a compact subset of $\T^3$. In other words, although we cannot prove that determining if a trajectory intersects an explicit open set is undecidable, it might be simply non-computable from a practical point of view because of the arbitrarily high computational cost of solving the problem\footnote{{More precisely, since the computational complexity of a problem that can be solved by a tape-bounded Turing machine tends to infinity as the tape size of the machine grows, we deduce that the reachability problem for a Beltrami field on $\mathbb T^3$ can be arbitrarily hard to solve in a robust manner.}} and the robustness of this property. It also follows from our construction that this phenomenon of robust computational complexity holds for fluid particle paths in a special class of time-dependent solutions of the Navier-Stokes equations on $\mathbb T^3$ (see Remark~\ref{rem:NS} in Section~\ref{S:bounded}).

The previous results together with the theory of Gaussian random Beltrami fields in Euclidean space~\cite{EPR} also allow us to prove that, in some sense, Turing completeness occurs with probability~$1$ in random Beltrami fields. So, in particular, the aforementioned kind of complexity is somehow \emph{generic}.

\begin{theorem}\label{T:gauss}
With probability~$1$, a Gaussian random Beltrami field in Euclidean space exhibits trajectories that simulate {every possible} tape-bounded Turing machine. In particular, these fields have trajectories of arbitrarily high computational complexity.
\end{theorem}
\begin{Remark}
Observe that in Theorem~\ref{T:gauss} every tape-bounded Turing machine is simulated by a single randomly chosen Beltrami field on $\mathbb R^3$ almost surely, whereas in Theorem~\ref{thm:torus}, for each tape-bounded machine there is a Beltrami field in $\mathbb T^3$ that simulates it.
\end{Remark}

We finish this introduction with a digression on the connections between computational and dynamical complexity. Turing completeness provides a measure of the complexity of a dynamical system. This puts computational complexity, undecidability, and dynamical complexity on the same footing. It is intriguing to understand how these different types of complexity are related. A usual measure of dynamical complexity is given by the topological entropy which, grosso modo, captures the exponential growth rate of the system. As a surprising spin-off of our constructions, not related to hydrodynamics, we show that there exist Turing complete smooth vector fields and diffeomorphisms on $\mathbb S^2$ with zero topological entropy. Albeit these dynamical systems are  not ``disordered'' enough to be considered dynamically complex, they are capable of universal computation. We find this result somehow surprising, and it connects in some sense with the study of the computational properties of the topological entropy of dynamical systems~\cite{Koiran}, or the entropy of dynamical models of Turing machines~\cite{Delvenne, GOT}.

\begin{theorem}\label{T:entropy}
There exist Turing complete smooth vector fields and diffeomorphisms on the sphere $\mathbb S^2$ with zero topological entropy.
\end{theorem}


\textbf{Organization of this article:} In Section~\ref{S:Turing} we recall the basics in the theory of Turing machines and their connections with dynamical systems. In particular, we recall the notion of Turing complete vector fields and the encoding of each configuration of a Turing machine as an open interval in $[0,1]$ introduced in~\cite{KCG}. In Section~\ref{S:CK} we prove a global version of the Cauchy-Kovalevskaya theorem for the curl operator established in~\cite{EP12}. In Section~\ref{S:TCgrad} we construct a planar Turing complete system of gradient form $\nabla_{\mathbb R^2} F$ whose computational power is weakly robust under perturbations. Theorem~\ref{th.main} is proved in Section~\ref{S:main}, while Theorem~\ref{thm:torus} on Beltrami fields in $\mathbb T^3$ and Theorem~\ref{T:gauss} on Gaussian random Beltrami fields are proved in Section~\ref{S:bounded}. Finally, in Section~\ref{S:final}, we prove Theorem~\ref{T:entropy}.
\\

\textbf{Acknowledgements:} We are indebted to Daniel Gra\c{c}a for enriching conversations that improved this article and to Crist\'obal Rojas for useful comments concerning the space-bounded Church-Turing thesis. We are also grateful to Michael Freedman for inquiring about the initial-condition robustness and the set-to-set formulation of Turing machine simulations, and to Leonid Polterovich for his question about topological entropy of our former construction~\cite{CMPP2}.

\section{Turing machines and dynamical systems}\label{S:Turing}

In this section, we recall the definition of Turing machine and how it can be simulated by a dynamical system. We also introduce an encoding that allows us to represent any state of the machine by an open interval in $[0,1]$. This interval encoding will be crucial later to construct a vector field in $\mathbb R^2$ that simulates a universal Turing machine.

\subsection{Basic definitions}\label{SS:basic}
A Turing machine $T$ is defined by the following data:
\begin{itemize}
\item A finite set $Q$ of ``states'' including an initial state $q_0$ and a halting state $q_{halt}$.
\item A finite set $\Sigma$ which is the ``alphabet'' with cardinality at least two.
\item A transition function $\delta:Q\times \Sigma \longrightarrow Q\times \Sigma \times \{-1,0,1\}$.
\end{itemize}
{We will work with Turing machines whose tape is always a string $t\in \Sigma^{\mathbb{Z}}$ with only a finite amount of symbols different from the blank symbol (that will be represented by zero). {This model of Turing machines is commonly used \cite{GCB, GCB2} and equivalent in computational power to other models.} It means that any tape is of the form
\begin{equation}\label{eq:string}
 ...00t_{-a}...t_{b}00...
\end{equation}
with $t_i\in\Sigma$ and $t_{-a}$ and $t_b$ are (respectively) the first and last digits different from zero of the tape. Of course, both $a\geq0$ and $b\geq0$ depend on the tape, so they are not uniformly bounded. In particular, at any given step, there are only (at most) $a+b+1$ non-blank symbols. The space of configurations of the machine $T$ is countable, and of the form $Q\times A \subset Q\times \Sigma^{\mathbb{Z}}$, where $A$ is the subset of strings of the form~\eqref{eq:string}. }

 For a given Turing machine $T=(Q,q_0,q_{halt},\Sigma,\delta)$ and an input tape $s=(s_n)_{n\in \mathbb{Z}}\in \Sigma^{\mathbb{Z}}$ the machine runs applying the following algorithm. {We will denote by $q$ and $t$ the current state and current tape respectively at a given step of the algorithm.}

\begin{enumerate}
\item Set the current state $q$ as the initial state and the current tape $t$ as the input tape.
\item If the current state is $q_{halt}$ then halt the algorithm and return $t$ as output. Otherwise compute $\delta(q,t_0)=(q',t_0',\varepsilon)$, with $\varepsilon \in \{-1,0,1\}$.
\item Replace $q$ with $q'$, and change the symbol $t_0$ by $t_0'$, obtaining the tape $\tilde t=...t_{-1}.t_0't_1...$ (as usual, we write a point to denote that the symbol at the right of that point is the symbol at position zero).
\item Shift $\tilde t$ by $\varepsilon$ obtaining a new tape $t'$, then return to step $(2)$ with current configuration $(q',t')$. Our convention is that $\varepsilon=1$ (resp. $\varepsilon=-1$) corresponds to the left shift (resp. the right shift).
\end{enumerate}

{A step of the algorithm can also be represented by a global transition function
$$ \Delta: Q\times A \longrightarrow Q\times A\,,$$
where we set $\Delta(q_{halt},t):=(q_{halt},t)$ for any tape $t$.	}
It is key to this work that a Turing machine can be simulated by a dynamical system (a vector field or a diffeomorphism). As usual, this is defined in terms of the halting problem for Turing machines. Specifically:

\begin{defi}\label{TC}
Let $X$ be a vector field on a manifold $M$. We say it is Turing complete if for any integer $k\geq 0$, given a Turing machine $T$, an input tape $t$, and a finite string $(t_{-k}^*,...,t_k^*)$ of symbols of the alphabet, there exist an explicitly constructible point $p\in M$ and an open set $U\subset M$ such that the orbit\footnote{Here by orbit we always refer to the ``trajectory", that is the orbit for positive times.} of $X$ through $p$ intersects $U$ if and only if $T$ halts with an output tape whose positions $-k,...,k$ correspond to the symbols $t_{-k}^*,...,t_k^*$.
\end{defi}

{Here ``explicitly constructible" means that the point and the open set are computable: they can be determined with arbitrary precision by a Turing machine (see~\cite{Wei} for an introduction to computable analysis).} Although it is not stated in the definition, most, if not all, constructions of Turing complete systems are done by encoding the ``step-by-step" evolution of a universal Turing machine in the phase space of a system. This is sometimes used as the notion of simulation of Turing machines \cite{GCB, GCB2}. In any case, the computational process or the halting problem becomes a ``point-to-set" property, where the initial point is always an explicitly constructible point. A ``set-to-set" property can also be considered. We will come back to this in Remark \ref{rem:settoset}. In the construction of Turing complete Beltrami fields of this article, the point $p$ depends only on the Turing machine $T$ and the input $t$, and the open set $U\equiv U_{t^*}$ depends on the given finite string $t^*:=(t_{-k}^*,...,t_k^*)$. This dependence of $U$ with $t^*$ is also in Tao's construction~\cite{T1} of a Turing complete diffeomorphism of $\mathbb T^4$, but not in our construction in~\cite{CMPP2}, where $U$ is fixed (it is related to the halting state of the Turing machine) and the point $p$ depends on all the information, i.e., the Turing machine $T$, the input tape $t$, and the finite string $t^*=(t_{-k}^*,...,t_k^*)$.

We want to emphasize that Definition~\ref{TC} and the well known undecidability of the halting problem for Turing machines imply that a Turing complete vector field $X$ exhibits undecidable long-term behavior. More precisely, it is undecidable to determine if the trajectory of $X$ through an explicit point will intersect an explicit open set of the space.

\subsection{An interval encoding}\label{SS.encoding}

Let us restrict to a special class of Turing machines, which is known to have the same computational power as a general Turing machine. Indeed, without any loss of generality, we may assume that the alphabet is $\Sigma=\{0,1,...,9\}$, where $0$ represents a special character referred to as the ``blank symbol", and $Q=\{1,...,m\}$, with $m$ the cardinality of the space of states $Q$.

Following~\cite{KCG}, let us now introduce an encoding of each configuration of $T$ as an open interval in $[0,1]$. To this end we define $s$ as the nonnegative integer whose digits are $t_{-a}...t_{-1}$ and $r$ as the nonnegative integer whose digits are $t_{b}...t_0$. Then each configuration $(q,t)$ can be identified with a rational point in $[0,1]$ by the map
\begin{align}\label{eq:IntEnc}
\begin{split}
\varphi: Q\times A &\longrightarrow [0,1]\\
 (q,t) &\longmapsto \frac{1}{2^q3^r5^s}\,.
 \end{split}
\end{align}

For any such a point $\alpha:=\varphi(q,t)=\frac{1}{2^q3^r5^s}$, {it is not hard to check that the intervals
$$I_{(q,t)}:=\Big(\alpha-\frac{\alpha^2}{8},\alpha+\frac{\alpha^2}{8}\Big)\,,$$
are pairwise disjoint. Indeed, it suffices to notice that
\[
\frac{1}{n}-\frac{1}{n+1}\geq \frac{1}{2n^2}
\]
for any positive integer $n$.} They obviously have size
{\begin{equation}\label{eq:intsize}
|I_{(q,t)}|=\alpha^2/4=\frac{1}{2^{2q+2}3^{2r}5^{2s}}\,.
\end{equation}}

\begin{Remark}\label{rem:haltdist}
A simple variation of this encoding allows us to represent the halting configurations $(q_{halt},t)$ by a set of points in $[0,1]$ that has a positive distance from the set of points associated with non-halting configurations $(q,t)$, $q\neq q_{halt}$. For example, this is the case if we set $$\varphi(q_{halt},t)=1-\frac{1}{2^{q_{halt}}3^r5^s}\,.$$
\end{Remark}

\section{A global Cauchy-Kovalevskaya theorem for Beltrami fields}\label{S:CK}

It is well known that the curl operator does not admit any non-characteristic surfaces, however in~\cite{EP12} it was shown that there is a version of the Cauchy-Kovalevskaya theorem for the Beltrami equation $\curl u=\lambda u$. Specifically, if $\Sigma\subset\mathbb R^3$ is an oriented analytic surface and $v$ is an analytic vector field tangent to $\Sigma$, there exists a unique Beltrami field $u$ in some neighborhood $N(\Sigma)$ of $\Sigma$ with Cauchy datum $u|_\Sigma=v$ if and only if
\begin{equation}\label{eq:ck}
d(j_\Sigma^* v^\flat)=0\,,
\end{equation}
where $j_\Sigma:\Sigma\to\mathbb R^3$ is the inclusion of $\Sigma$ into $\mathbb R^3$ and $v^\flat$ is the $1$-form dual to the vector field $v$ (using the Euclidean metric).

Our goal in this section is to prove a global version of the aforementioned Cauchy-Kovalevskaya theorem for the curl operator. More precisely, we show that when the Cauchy surface $\Sigma$ is an affine plane and the tangent Cauchy datum $v$ is entire, the Cauchy problem has a unique solution $u$ that is defined on the whole $\mathbb R^3$ (and is entire as well). We recall that an analytic function on $\mathbb{R}^n$ is entire if it can be extended to a holomorphic function on $\mathbb{C}^n$.

Parametrizing $\mathbb R^3$ with Cartesian coordinates $(x,y,z)$, let us assume that
$$\Sigma=\{z=0\}\,.$$
In these coordinates, the analytic Cauchy datum $v$ tangent to $\Sigma$ can be written as
\[
v(x,y)=v_1(x,y)\partial_x+v_2(x,y)\partial_y\,.
\]
It is then immediate to check that Equation~\eqref{eq:ck} reads as
\[
\frac{\partial v_1}{\partial y}-\frac{\partial v_2}{\partial x}=0
\]
for all $(x,y)\in\mathbb R^2$, so we conclude that $v$ is a gradient field, i.e., there exists an analytic function $F:\mathbb R^2\to\mathbb R$ such that
$$v(x,y)=\nabla_{\mathbb R^2} F:=\frac{\partial F}{\partial x}\partial_x+\frac{\partial F}{\partial y}\partial_y\,.$$
In particular, this implies that any Beltrami field that is tangent to the plane $\{z=0\}$ is the gradient of an analytic function on the plane.

We can now state the main theorem of this section in terms of gradient fields as Cauchy data.

\begin{theorem}\label{thm:main1}
Let $F(x,y)$ be an entire function on the plane $\{z=0\}$. Then for any $\lambda\neq 0$ there exists a (unique) Beltrami field $u\in \mathbb{R}^3$ that solves the Cauchy problem
\begin{equation*}
\curl u=\lambda u\,, \qquad u|_{\{z=0\}}=\nabla_{\mathbb R^2} F\,.
\end{equation*}
\end{theorem}

\begin{proof}
Let us first consider the auxiliary problem
\begin{equation}\label{eq:aux}
\Delta v=-\lambda^2 v
\end{equation}
for a vector field $v=v_1\partial_x+v_2\partial_y+v_3\partial_z$. The action of the Laplacian $\Delta$ on $v$ is understood componentwise. Since we want to prescribe Cauchy data on the plane $\{z=0\}$, it is convenient to write the Helmholtz equation in Cauchy form:

\begin{equation}\label{eq1}
\frac{\partial^2v}{\partial z^2}=-\lambda^2 v-\frac{\partial^2v}{\partial x^2}-\frac{\partial^2v}{\partial y^2}\,.
\end{equation}

We can then apply the Cauchy-Kovalevskaya theorem to Equation~\eqref{eq1} with Cauchy data:

\begin{equation}\label{eq:datum}
\begin{cases}
v|_{\{z=0\}}&=\frac{\partial F}{\partial x}\partial_x+\frac{\partial F}{\partial y}\partial_y\,,\\
\frac{\partial v}{\partial z}\Big|_{\{z=0\}}&=\lambda\frac{\partial F}{\partial y}\partial_x -\lambda \frac{\partial F}{\partial x}\partial_y-\Delta_{\mathbb R^2} F\partial_z\,,
\end{cases}
\end{equation}
where $\Delta_{\mathbb R^2}:=\frac{\partial^2}{\partial x^2}+\frac{\partial^2}{\partial y^2}$ stands for the Laplacian on $\mathbb R^2$.

Since $F$ is an entire function, a straightforward application of the global Cauchy-Kovalevskaya theorem~\cite[Theorem 1.1]{PW} to each component $v_k$, $k=1,2,3$, implies that there exists a unique entire vector field $v$ in $\mathbb{R}^3$ solving Equation~\eqref{eq1} with Cauchy data~\eqref{eq:datum}.

We claim that the vector field $v$ is divergence-free, i.e., $\dive v=0$ in $\mathbb R^3$. First, since the operators $\Delta$ and $\dive$ commute, it is obvious that $\dive v$ satisfies the Helmholtz equation
$$\Delta (\dive v)=-\lambda^2 (\dive v)$$
in $\mathbb R^3$. Now, observe that the Cauchy data~\eqref{eq:datum} imply
\begin{align*}
\dive v|_{\{z=0\}}&= \Big(\frac{\partial v_1}{\partial x}+\frac{\partial v_2}{\partial y}+\frac{\partial v_3}{\partial z}\Big)\Big|_{\{z=0\}}=
\frac{\partial^2 F}{\partial x^2}+\frac{\partial^2 F}{\partial y^2}-\Delta_{\mathbb R^2} F\\
&=0\,,
\end{align*}
and
\begin{align*}
\frac{\partial \dive v}{\partial z}\Big|_{\{z=0\}}&=\Big(\frac{\partial^2 v_1}{\partial x\partial z}+\frac{\partial^2 v_2}{\partial y\partial z}+\frac{\partial^2 v_3}{\partial z^2}\Big)\Big|_{\{z=0\}}\\
&=\lambda \frac{\partial^2 F}{\partial x\partial y}-\lambda \frac{\partial^2 F}{\partial x\partial y}- \lambda^2 v_3|_{\{z=0\}}-\frac{\partial^2 v_3}{\partial x^2}\Big|_{\{z=0\}}-\frac{\partial^2 v_3}{\partial y^2}\Big|_{\{z=0\}}\\
&=0\,,
\end{align*}
where we have used that $v_3$ satisfies the Helmholtz equation~\eqref{eq1} and that $v_3|_{\{z=0\}}=0$.

Since $\dive v$ is analytic, a straightforward application of the Cauchy-Kovalevskaya theorem for the Helmholtz equation then implies that $\dive v=0$ in $\mathbb R^3$, as claimed.

Next, using the vector calculus identity $\curl \curl =\nabla \dive - \Delta$, we easily infer from Equation~\eqref{eq:aux} that the divergence-free vector field $v$ satisfies the following equation
\begin{equation}\label{eq:ll}
(\curl -\lambda)(\curl+\lambda) v=0\,.
\end{equation}

Finally, the desired vector field is defined as
$$u:=\frac{1}{2\lambda}(\curl +\lambda)v\,.$$
Indeed, it follows from Equation~\eqref{eq:ll} that the vector field $u$ is a Beltrami field,
$$(\curl -\lambda) u=0$$
in $\mathbb{R}^3$, which is entire because so is $v$. To check that $u|_{\{z=0\}}=\nabla_{\mathbb R^2} F$ we simply notice that
\begin{align*}
2\lambda u|_{\{z=0\}}&=\Big(\lambda v_1+\frac{\partial v_3}{\partial y}-\frac{\partial v_2}{\partial z}\Big)\Big|_{\{z=0\}}\partial_x+ \Big(\lambda v_2+\frac{\partial v_1}{\partial z}-\frac{\partial v_3}{\partial x}\Big)\Big|_{\{z=0\}}\partial_y\\
&+\Big(\lambda v_3+\frac{\partial v_2}{\partial x}-\frac{\partial v_1}{\partial y}\Big)\Big|_{\{z=0\}}\partial_z\\
&=\Big(\lambda\frac{\partial F}{\partial x}+\lambda\frac{\partial F}{\partial x}\Big)\partial_x+\Big(\lambda\frac{\partial F}{\partial y}+\lambda\frac{\partial F}{\partial y}\Big)\partial_y+\Big(\frac{\partial^2 F}{\partial x\partial y}-\frac{\partial^2 F}{\partial x\partial y}\Big)\partial_z\\
&=2\lambda \nabla_{\mathbb R^2}F\,,
\end{align*}
thus completing the proof of the theorem.
\end{proof}

%
%
%
%
%
%

\section{Construction of a Turing complete gradient field in $\mathbb R^2$} \label{S:TCgrad}

In this section, we prove that there exists a smooth gradient vector field $X$ on the plane that can simulate a universal Turing machine. Moreover, this computational power is \emph{weakly robust} in the sense that there exists an error function $\epsilon:\mathbb R^2\to (0,1)$ (tending to zero at infinity fast enough) such that any smooth vector field $Y$ that is pointwise close to $X$ as
\begin{equation}\label{eq_wrob}
|X(x,y)-Y(x,y)|\leq \epsilon(x,y)
\end{equation}
for all $(x,y)\in\mathbb R^2$, is also Turing complete.

\begin{theorem}\label{th.turingrob}
There exists a $C^\infty$ function $f:\mathbb R^2\to \mathbb R$ such that its gradient $\nabla_{\mathbb R^2}f$ is a weakly robust Turing complete vector field.
\end{theorem}

All along this section we use the notation and results presented in Section~\ref{SS.encoding} without further mention, and we assume that $T=(Q, q_0,q_{halt}, \Sigma, \delta)$ is a universal Turing machine (although all our constructions work for an arbitrary Turing machine). For the sake of simplicity, we shall use the same notation $c$ for a configuration $(q,t)$ and its encoding as a point $\varphi(q,t)$ in $[0,1]$ (and the same for $\Delta(c)$, $\Delta^2(c)$, etc.). Additionally, if $I_{(q,t)}=:(a,b)\subset [0,1]$ is the interval that encodes $(q,t)$, we denote by
$$I_{(q,t)}^j:= (a+2j,b+2j)$$
the translation of $I_{(q,t)}$ contained in $[2j,2j+1]$.

Since we want to construct a planar vector field that simulates a Turing machine, it is convenient to introduce a planar encoding that relates configurations of the machine with open sets of $\mathbb{R}^2$:
\begin{defi}\label{D:ep}
Let $\varepsilon>0$ be any small but fixed constant. To each configuration $(q,t)$ we assign the open set
\begin{equation}\label{eq:encod}
(q,t) \longmapsto \bigcup_{j,k=0}^\infty I_{(q,t)}^j \times (k-\varepsilon/2,k+\varepsilon/2) =: U_{(q,t)}\subset \mathbb R^2\,,
\end{equation}
which consists of a countable union of pairwise disjoint squares. It is clear that $U_{(q,t)}\cap U_{(q',t')}=\emptyset$ if $(q,t)\neq (q',t')$.
\end{defi}

\subsection{Step~1: a good system of encoding curves}\label{SS.curves}

We aim to construct a countable family of planar unbounded curves (pairwise disjoint) that encodes (in some sense that will become clear later) the evolution of the Turing machine.

Denoting by $C_0$ the space of initial configurations $\{q_0\} \times A$ of the machine, it is clear that this set is ordered by the identification with points in $[0,1]$, i.e., $C_0=\{c_i\}_{i=0}^\infty$, $c_0>c_1>...$ (this is equivalent to ordering the configurations increasingly as integers with respect to their representation of the form $2^q3^r5^s$). Each configuration $c_i$ has its associated orbit by the global transition function: $c_i, \Delta(c_i), \Delta^2(c_i),...$. We observe that this orbit is infinite because of the way we have extended the transition function to act on halting states, which implies that if $\Delta^n(c_i)$ is a halting configuration, then $\Delta^j(c_i)=\Delta^n(c_i)$ for all $j\geq n$.

We are now ready to associate to each initial configuration $c_i$, $i\geq 0$, an infinite sequence of points $\{p^i_0, p^i_1,...\}\subset\mathbb R^2$:
$$
p^i_l:=(\Delta^l(c_i)+2i,l)\in\mathbb R^2\,,
$$
i.e., $p^i_l$ is the point whose $x$-coordinate is the number in the middle of the interval $I_{\Delta^l(c_i)}^i$ and its $y$-coordinate is $l$. Roughly speaking, $p^i_l$ corresponds to the configuration of the machine after $l$ steps with input $c_i$, taken in the ``band" number $i$ and at height $l$.

\begin{figure}[!h]
\begin{center}
\begin{tikzpicture}
     \node[anchor=south west,inner sep=0] at (0,0) {\includegraphics[scale=0.18]{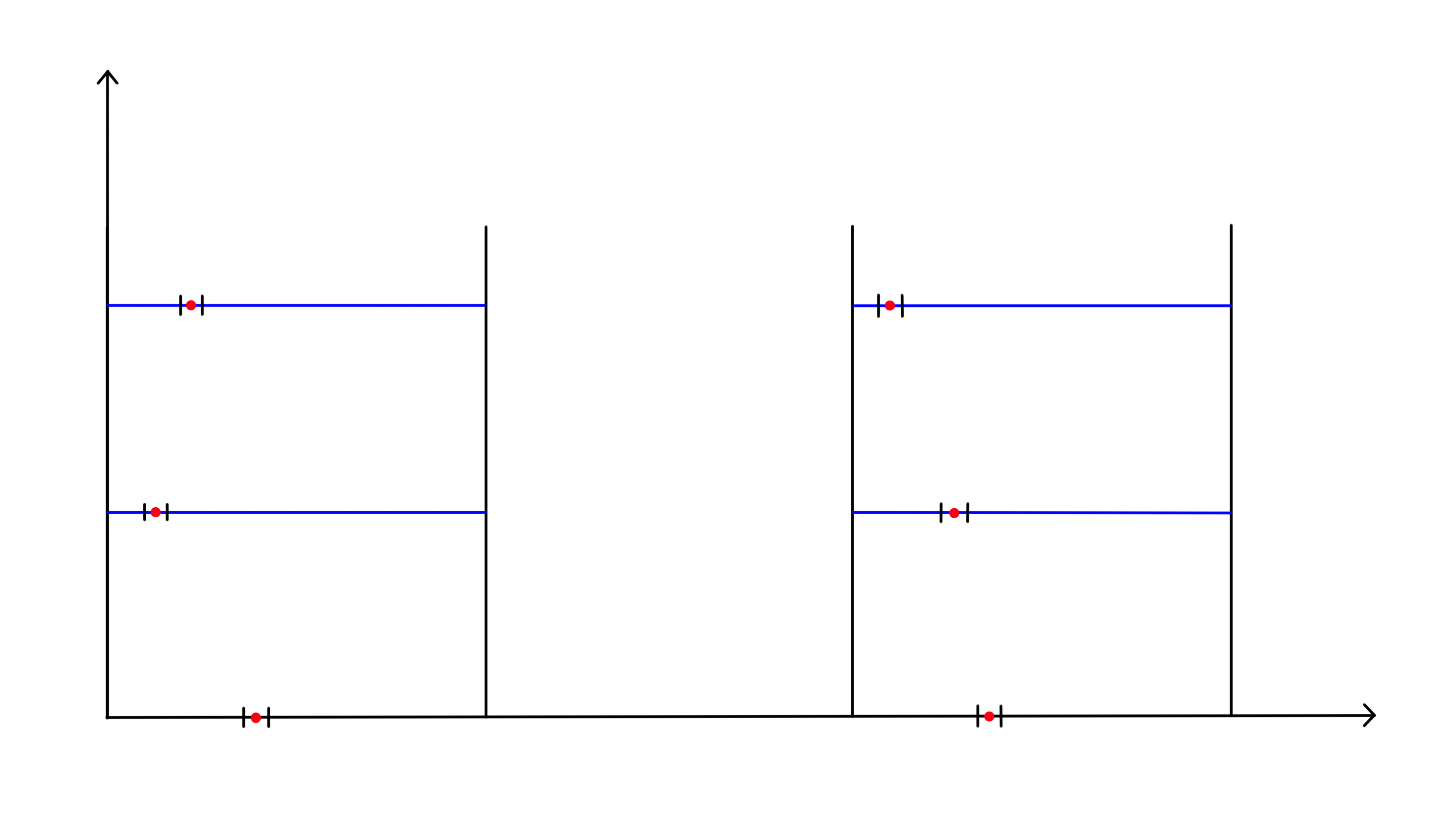}};

     \node[red] at (2.45,1.55) {$p_0^0$};
     \node[red] at (1.55,2.7) {$p_1^0$};
     \node[red] at (1.8, 4.7) {$p_2^0$};

     \node[red] at (9.5,0.68) {$p_0^1$};
     \node[red] at (9.2,2.7) {$p_1^1$};
     \node[red] at (8.6, 4.7) {$p_2^1$};

     \node [scale=0.8] at (2.6,0.68) {$I_{c_0}$};

     \draw [thick] (1.418, 1.04)-- (1.418,1.24);
     \draw [thick] (1.418+0.2, 1.04)-- (1.418+0.2,1.24);

     \node [scale=0.8] at (1.63, 0.68) {$I_{\Delta(c_0)}$};

     \node at (13.5,0.6) {$x$};
     \node at (0.5,7.5) {$y$};

     \node at (4.65,0.7) {$1$};
     \node at (8.2,0.7) {$2$};
     \node at (11.85,0.7) {$3$};

     \draw [thick] (4.673,1.04) -- (4.673, 1.17);
     \draw [thick] (8.201,1.04) -- (8.201, 1.17);
     \draw [thick] (11.84,1.04) -- (11.84, 1.17);

     \draw [thick] (0.95,3.13)--(1.05,3.13);
     \node at (0.65,3.12) {$1$};

\end{tikzpicture}
\caption{Points representing the first steps of the algorithm with inputs $c_0$ and $c_1$}
\label{fig:points}
\end{center}
\end{figure}
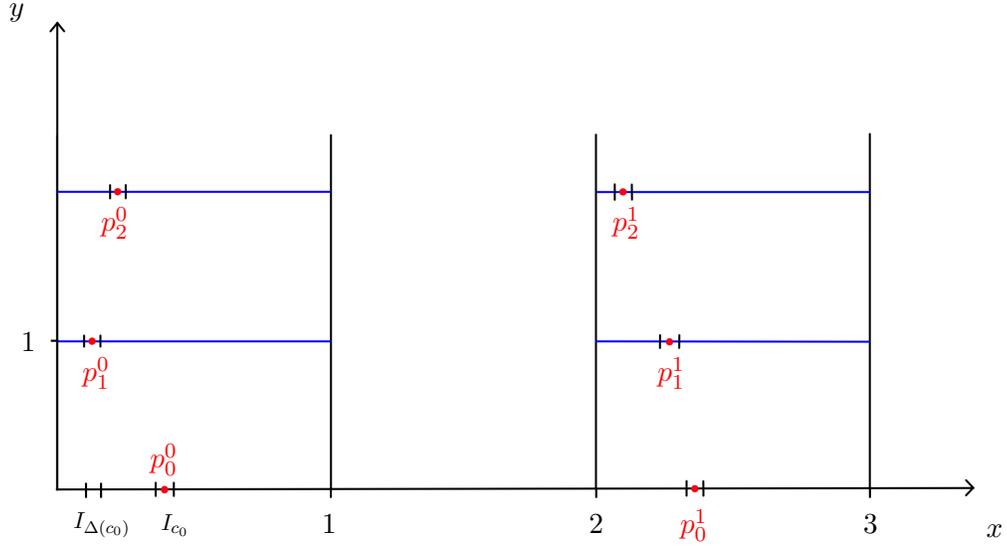

Finally, we can construct a well-behaved family of pairwise disjoint smooth curves $\{\gamma_i\}_{i=0}^\infty$ that interpolates all the points $\{p^i_l\}$:

\begin{lemma}\label{L:curves}
For each $i\geq 0$ there exists an open $C^\infty$ curve $\gamma_i$ that is properly embedded in $[2i,2i+1] \times (-\infty,\infty)$ and satisfies:
\begin{enumerate}
\item The (unit) tangent vector has a nonzero $y$-projection at each point of $\gamma_i$, which is bounded from below by a positive constant on the whole curve.
\item It contains all the points $\{p^i_l\}_{l=0}^\infty$.
\item The curvature function $\kappa_i$ is upper bounded:
\[
|\kappa_i(p)|<15
\]
for all $p\in\gamma_i$.
\item $\gamma_i$ is vertical in an $\varepsilon$-neighborhood of $p^i_l$, i.e., it is of the form
\[
\{(\Delta^l(c_i)+2i,y): y\in(l-\varepsilon,l+\varepsilon)\}
\]
for all $i\geq0$ and $l\geq0$. Here $\varepsilon>0$ is the small constant introduced in Definition~\ref{D:ep}.
\end{enumerate}
In particular, $\gamma_i$ has infinite length.
\end{lemma}
\begin{proof}
We can describe explicitly each curve $\gamma_i$ in terms of a parametrization that we denote by $\gamma_i(s)$, $s\in\mathbb R$ (notice that $s$ is not arc-length here). Indeed, the curve is given by
\[
\gamma_i(s)=(\lambda_i(s),s)\,,
\]
where the function $\lambda_i(s)$ is defined as
\begin{align*}
\lambda_i(s):=\begin{cases}
c_i+2i \qquad \text{if } s\leq \varepsilon\,,\\
\Delta^l(c_i)+2i \qquad \text{if } |s-l|\leq \varepsilon
\end{cases}
\end{align*}
for all $l\geq 1$, and
$$
\lambda_i(s)=\frac{\int_{l+\varepsilon}^s e^{\frac{-1}{s'-l}}e^{\frac{-1}{l+1-s'}}ds'}{\int_{l+\varepsilon}^{l+1-\varepsilon} e^{\frac{-1}{s'-l}}e^{\frac{-1}{l+1-s'}}ds'}\Big(\Delta^{l+1}(c_i)-\Delta^l(c_i)\Big)+2i+\Delta^l(c_i)
$$
if $s\in (l+\varepsilon,l+1-\varepsilon)$ and $l\geq0$.
Obviously $\lambda_i(s)$ is a $C^\infty$ function. By construction $\gamma_i(l)=p^i_l$ and $\gamma'_i(s)=(\lambda_i'(s),1)$, so its unit tangent vector $T_i(s)$ has a $y$-component for all $s\in\mathbb R$ given by
\[
T_i(s)\cdot\partial_y=\frac{1}{(1+\lambda'_i(s)^2)^{1/2}}\,,
\]
whose infimum on $s\in\mathbb R$ is positive. In particular, $\gamma_i$ is vertical in an $\varepsilon$-neighborhood of each $p^i_l$.

Moreover, since
\[
\lambda'_i(s)=\frac{e^{\frac{-1}{s-l}}e^{\frac{-1}{l+1-s}}}{\int_{l+\varepsilon}^{l+1-\varepsilon} e^{\frac{-1}{s'-l}}e^{\frac{-1}{l+1-s'}}ds'}\Big(\Delta^{l+1}(c_i)-\Delta^l(c_i)\Big)
\]
if $s\in(l+\varepsilon,l+1-\varepsilon)$ (and it is $0$ otherwise), we infer that $\lambda'_i$ has a constant sign in each interval $(l+\varepsilon,l+1-\varepsilon)$, given by the sign of the difference $\Delta^{l+1}(c_i)-\Delta^l(c_i)$. Using that $\Delta^l(c_i)\in [2i,2i+1]$ for all $l$, we conclude that the curve $\gamma_i$ is contained (and properly embedded) in the band $[2i,2i+1] \times (-\infty,\infty)$.

Finally, to check that the curvature is uniformly bounded, we notice that
$$
\kappa_i(s)=-\lambda_i''(s)\Big(1+(\lambda'_i(s))^2\Big)^{-3/2}\,,
$$
where
$$
\lambda''_i(s)=\frac{(\Delta^{l+1}(c_i)-\Delta^l(c_i))e^{\frac{-1}{s-l}}e^{\frac{-1}{l+1-s}}}{\int_{l+\varepsilon}^{l+1-\varepsilon} e^{\frac{-1}{s'-l}}e^{\frac{-1}{l+1-s'}}ds'}\cdot\bigg(\frac{1}{(s-l)^2}-\frac{1}{(l+1-s)^2}\bigg)
$$
if $s\in (l+\varepsilon,l+1-\varepsilon)$, and it is $0$ otherwise. In order to obtain an upper bound for $|\lambda''_i|$, we can reduce all the computations to the interval $s\in[\varepsilon,1-\varepsilon]$ after a translation $s\to s+l$. Therefore, noticing that
$$
\int_0^1 e^{\frac{-1}{s'}}e^{\frac{-1}{1-s'}}ds'>7\cdot 10^{-3}
$$
and
\begin{align*}
\text{sup}_{s\in(0,1)} \bigg(e^{\frac{-1}{s}}e^{\frac{-1}{1-s}}\cdot\Big(\frac{1}{s^2}-\frac{1}{(1-s)^2}\Big)\bigg)<10^{-1}\,,
\end{align*}
we deduce that if $\varepsilon$ is small enough
\[
|\kappa_i(s)|<\frac{100}{7}<15\,,
\]
for all $i\geq0$ and $s\in\mathbb R$, where we have used that $|\Delta^{l+1}(c_i)-\Delta^l(c_i)|\leq 1$. This completes the proof of the lemma.
\end{proof}

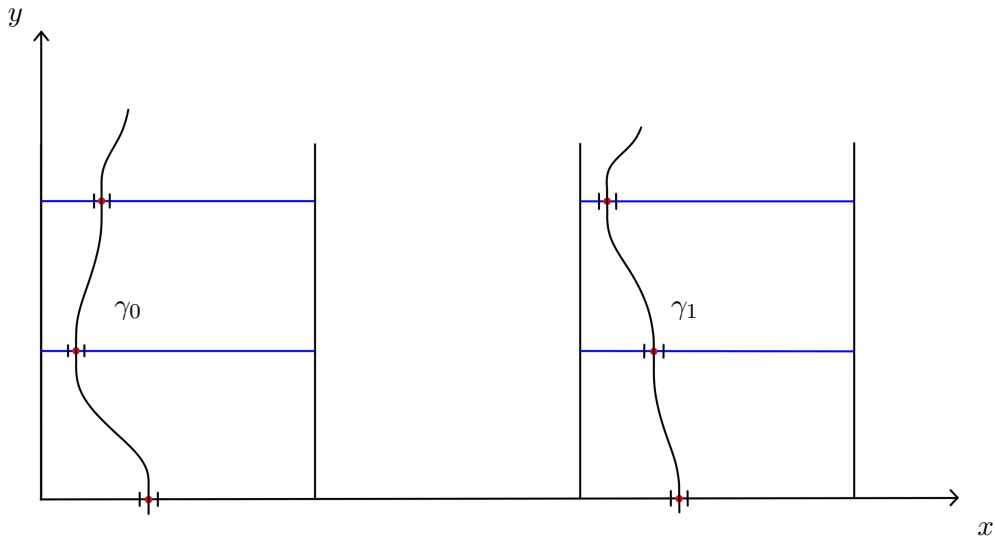
\begin{figure}[!h]
\begin{center}
\begin{tikzpicture}
     \node[anchor=south west,inner sep=0] at (0,0) {\includegraphics[scale=0.18]{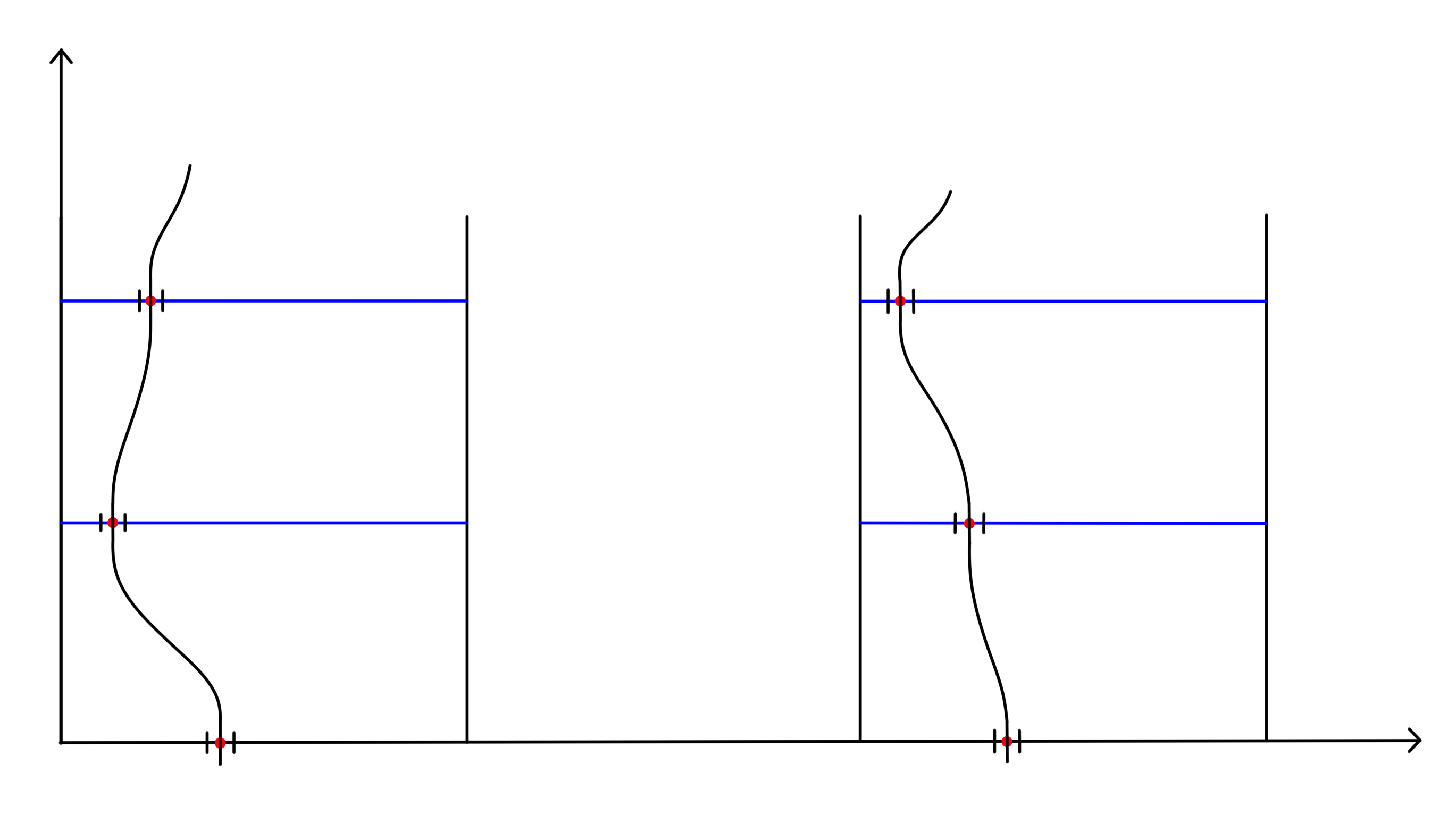}};

               \node at (13.1,0.4) {$x$};
     \node at (0.2,7.2) {$y$};

     \node at (1.7,3.3) {$\gamma_0$};
      \node at (9.1,3.3) {$\gamma_1$};

\end{tikzpicture}
\caption{Curves following the orbits of the Turing machine}
\label{fig:curves}
\end{center}
\end{figure}

\subsection{Step~2: Construction of the function $f$}\label{SS.grad}

In the previous section, we have constructed a well-behaved family $\{\gamma_i\}_{i=0}^\infty$ of smooth curves that interpolate all the points $p^i_l$ associated with the dynamics of the Turing machine. Our aim now is to construct a function $f\in C^\infty(\mathbb R^2)$ such that the vector field $\nabla_{\mathbb R^2}f$ is tangent to the curves $\gamma_i$ and satisfies an additional contraction property that will be key in Step~3 to prove the weak robustness of the computational power of $\nabla_{\mathbb R^2} f$.

To this end we introduce a ``thickening'' $B_i$ (a metric neighborhood) of each curve $\gamma_i$ defined as
$$
B_i:=\Big\{p\in\mathbb R^2: \text{dist}(p,\gamma_i)<\frac{1}{16}\Big\}\,.
$$
We notice that this neighborhood is well defined and diffeomorphic to a product $\mathbb R\times (-1/16,1/16)$ because the injectivity radius of $\gamma_i$ is lower bounded by $\frac{1}{15}$ on account of Lemma~\ref{L:curves}. It also holds that $B_i$ contains the boxes $I^i_{\Delta^l(c_i)}\times (l-\varepsilon,l+\varepsilon)$ for all $l\geq 0$ because the curve $\gamma_i$ is vertical in such a region (cf. Lemma~\ref{L:curves}) and $I^i_{\Delta^l(c_i)}\times\{l\}$ is a horizontal interval centered at $p^i_l$ of radius $\alpha^2/8$, which is at most $1/32$ (because $\alpha\leq 1/2$).

Let us now parametrize each neighborhood $B_i$ with coordinates
$$(s,\rho)\in \mathbb R\times (-1/16,1/16)\,,$$
where $s$ is a signed arc-length parametrization of $\gamma_i$ (taking $\gamma_i(0)=(c_i+2i,0)$) and $\rho$ is the signed distance function to $\gamma_i$, which is $C^\infty$, see e.g.~\cite{KP}. Using the unit normal vector $N_i(s)$ along $\gamma_i$, we can write the parametrization as
\begin{align*}
\Phi: (-\infty, \infty) \times (-1/16,1/16) & \longmapsto B_i\\
		(s,\rho) &\longmapsto \gamma_i(s) + \rho N_i(s).
\end{align*}
In these coordinates, the Euclidean metric reads as
$$ g_0=d\rho^2+\left(1-\kappa_i(s)\rho\right)ds^2, $$
where $\kappa_i(s)$ is the (signed) curvature of $\gamma_i$.

\begin{Remark}\label{R:tran}
Since $|\kappa_i(s)\rho|< \frac{15}{16}$ in the band $B_i$, we conclude that the vector field
\[
\partial_s=(1-\kappa_i(s)\rho)\gamma'_i(s)
\]
is transverse to the lines $\{y=\text{constant}\}$ on the whole band.
\end{Remark}

Next we consider a function $f_i$ of the form
\begin{align*}
f_i: B_i &\longrightarrow \mathbb{R}\\
	(s,\rho) &\longmapsto H_i(s) -\frac{\rho^2}{2}\,,
\end{align*}
where $H_i(s)$ is a function (that we will choose later) whose derivative $\Lambda_i(s):=H_i'(s)>0$ we assume to be bounded and positive for all $s\in\mathbb R$. The gradient vector field of $f_i$ in each neighborhood $B_i$ is
$$ X_i:=\nabla_{\mathbb R^2} f_i= \frac{\Lambda_i(s)}{1-\kappa_i(s)\rho}\partial_s- \rho \partial_\rho\,.$$
It is clear that $X_i$ is tangent to the curve $\gamma_i$ and contracts in the normal direction. In particular, the estimate
\[
X_i\cdot \nabla \rho^2=-2\rho^2\leq 0
\]
implies that any neighborhood $\{\rho<\text{constant}\}$ is invariant under the forward flow of $X_i$, so in particular the band $B_i$. Moreover, since $\Lambda_i(s)>0$, the vector field $X_i$ is transverse to the lines $\{s=\text{constant}\}$, and also to the lines $\{y=\text{constant}\}$ in each box $B_i\cap \Big(\mathbb R\times(l-\varepsilon,l+\varepsilon)\Big)$.

%

Next, we show that the function $\Lambda_i(s)$ can be chosen so that $X_i$ strongly contracts certain subintervals of $I^i_{\Delta^l(c_i)}$, a property that will be used in Section~\ref{SS.wr}. First, we need the following technical lemma that gives a lower bound for the size of each interval $|I_{\Delta^l(c_i)}^i|$.

\begin{lemma}\label{lem:intsize}
The following lower bound holds for all $i,l$:
$$ |I_{\Delta^l(c_i)}^i| > { \frac{1}{2^{2m+2}3^{2\cdot10^{i+l+1}}5^{2\cdot 10^{i+l+1}}} }=:A_{i,l}\,. $$

\end{lemma}
\begin{proof}
Recall that the point representation of an input $c_i$ is of the form
$$\frac{1}{2^{q_0}3^{r}5^{s}} \in [0,1]\,,$$
so according to the ordering that we have introduced, the input tape of $c_i$ is such that the integer $2^{q_0}3^{r}5^{s}$ is the $i^{th}$ smallest number within all possible input tapes. By definition of the integers $r$ and $s$, $c_i$ roughly depends on the number of digits of the tape at left and right of the zeroth position. For example, the first input $c_0$ has an empty tape (i.e., $r=s=0$), and the next input $c_1$ corresponds to $s=0, r=1$. It should be clear from this ordering that the input $c_i$ is associated with integers $r$ and $s$ that have each at most $i$ digits (in fact less in general, but we avoid giving a sharp estimate as it is not necessary for our discussion). This proves that a lower bound on the size of the interval $I^i_{c_i}$ is given by
\begin{equation}\label{eq:inpsize}
   |I^i_{c_i}| > {\frac{1}{2^{2q_0+2}3^{2\cdot10^{i+1}}5^{2\cdot10^{i+1}}}\,.}
 \end{equation}

The size of $I^i_{\Delta^l(c_i)}$ depends on the configuration $\Delta^l(c_i)$.  We can argue as before to analyze how much the size of an interval can be decreased after $l$ steps of the algorithm, which is equivalent to studying how much the representation in $[0,1]$ becomes smaller. Whatever the state of the configuration $\Delta^l(c_i)$ is, it will always be a number smaller or equal than $m$, since this is the maximal number in the set $Q$. As shown before, $c_i$ has a tape such that $r$ and $s$ have at most $i$ digits. Since at each step of the algorithm there is a single $\varepsilon$-shift, with $\varepsilon \in \{-1,0,1\}$, we infer that the numbers $r$ and $s$ will have their digit lengths increased by one at most in each step. Hence after $l$ steps of the algorithm, the numbers representing the left and right sides of the tape will have at most $i+l$ nonzero digits. Therefore, we conclude that
$$  |I_{\Delta^l(c_i)}^i| > {\frac{1}{2^{2m+2}3^{2\cdot10^{i+l+1}}5^{2\cdot10^{i+l+1}}}\,,} $$
as we wanted to prove.
\end{proof}

We are now ready to prove the aforementioned strongly contracting property of $\nabla_{\mathbb R^2}f_i$. In the statement, we denote by $s^i_l$ the distance between the points $p^i_0$ and $p^i_l$ measured along the curve $\gamma_i$. If $\varepsilon$ is the small constant introduced before, it is evident from the properties of our construction that
\[
\{(s,\rho):|s-s^i_l|<\varepsilon,|\rho|<1/16\}=B_i\cap \{|y-l|<\varepsilon\}\,.
\]

\begin{prop}\label{prop:contrac}
Assume that $\Lambda_i(s)\leq\frac{K_i}{16C}e^{-Cs}$ where $K_i:= \frac{1}{\ln(15^{3\cdot10^{i+2}})}$ and $C:=\ln 10$.
Then for each $i\geq0$ and $l\geq 0$ the integral curves $(s(t),\rho(t))$ of $X_i$ with initial data $s(0)=s^i_l$ and $|\rho(0)|<\frac{A_{i,l}}{2^j}$ for some $j\geq 1$ satisfy
\[
|\rho(t)|<\frac{A_{i,l+1}}{2^{j+1}}
\]
whenever $|s(t)-s^i_{l+1}|<\varepsilon$.
\end{prop}

\begin{proof}
The trajectory $(s(t),\rho(t))$ of $X_i$ satisfies the differential equation
\begin{equation}\label{eq:gradientflow}
\begin{cases}
\frac{ds}{dt}=\frac{\Lambda_i(s)}{1-\kappa_i(s(t))\rho(t)}\,,\\
\frac{d\rho}{dt}=-\rho(t)\,,
\end{cases}
\end{equation}
with initial condition $s(0)=s^i_l$ and $|\rho(0)|<A_{i,l}/2^j$. Since $\kappa_i(s)$ is upper bounded by $15$ and $|\rho(t)|<1/16$ by the forward invariance of the neighborhood $B_i$, we can write
\begin{equation}\label{eq:fastestflow}
\begin{cases}
\frac{ds}{dt}<16\Lambda_i(s)\leq \frac{K_i}{C}e^{-Cs}\,,\\
\frac{d\rho}{dt}=-\rho(t)\,,
\end{cases}
\end{equation}
where $K_i$ and $C$ are the constants introduced in the statement. Integrating this differential inequality we obtain
\begin{equation}\label{eq:rhofast}
\begin{cases}
s(t)< \frac{1}{C}\ln(K_it+e^{Cs^i_l})\,,\\
\rho(t)= \rho(0)e^{-t}\,.
\end{cases}
\end{equation}
A straightforward computation then shows that
$$t>\frac{e^{Cs(t)}}{K_i}-\frac{e^{Cs^i_l}}{K_i}\,.$$
If we denote by $T_{i,l}(c)$ the time needed for the trajectory to reach the position $s(T_{i,l}(c))=s^i_{l+1}+c$, with $|c|<\varepsilon$, we easily obtain the bound
\begin{align*}
T_{i,l}(c)&=\frac{e^{C(s^i_{l+1}+c)}-e^{Cs^i_l}}{K_i} \\
				&> \frac{e^{C(s^i_l+\frac{1}{2})}-e^{Cs^i_l}}{K_i}\\
				&= \frac{e^{Cs^i_l}(e^{\frac{C}{2}}-1)}{K_i} \\
				&> \frac{e^{Cl}}{K_i}\,,
\end{align*}
where we have used that $s^i_{l+1}-s^i_l+c>1-\varepsilon>1/2$ because the distance between the points $p^i_{l+1}\in\{y=l+1\}$ and $p^i_{l}\in\{y=l\}$ is at least $1$, $s^i_l>l$ by the same reason, and that we will choose a constant $C>2$. This bound immediately yields the following estimate for $\rho(T_{i,l}(c))$:

\begin{equation}\label{eq_rho}
|\rho(T_{i,l}(c))|<\frac{A_{i,l}}{2^j} e^{-\frac{1}{K_i}e^{Cl}}\,.
\end{equation}
Noticing that
$$\frac{A_{i,l}}{A_{i,l+1}}= 15^{2.10^{i+l+2}-2.10^{i+l+1}} \implies A_{i,l}< 15^{2.10^{i+l+2}}A_{i,l+1}\,,$$
and using Equation~\eqref{eq_rho} and the definitions of the constants $K_i$ and $C$, we finally obtain the bound
\begin{align*}
|\rho(T_{i,l}(c))|&<\frac{15^{2\cdot10^{i+l+2}}}{2^j} A_{i,l+1}e^{-\frac{1}{K_i}e^{Cl}}\\
				&= \frac{A_{i,l+1}}{2^j} \frac{1}{15^{10^{i+l+2}}}\\
				&< \frac{A_{i,l+1}}{2^{j+1}},
\end{align*}
for all $i\geq0$, $l\geq0$ and $|c|<\varepsilon$, thus completing the proof of the Proposition.
\end{proof}

Accordingly, in what follows we shall choose functions $f_i=H_i(s)-\frac{\rho^2}{2}$ where $H_i(s)$ is such that $H'_i(s):=\Lambda_i(s) \leq \frac{K_i}{16C}e^{-Cs}$ and bounded.

\begin{Remark}
If $(x(t),y(t))$ denotes the integral curves in Proposition~\ref{prop:contrac} in Cartesian coordinates, its statement is equivalent to:
\[
\Big|x(t)-(\Delta^{l+1}(c_i)+2i)\Big|<\frac{A_{i,l+1}}{2^{j+1}}
\]
provided that $|y(t)-(l+1)|<\varepsilon$.
\end{Remark}

Finally, we can define the function $f\in C^\infty(\mathbb R^2)$ as any smooth extension of the functions $f_i$, that is
\[
f|_{B_i}=f_i
\]
for all $i\geq0$. This smooth extension can be done explicitly taking, for example, a nonnegative $C^\infty$ bump function $\chi_i(\rho)$ on each band $B_i$ such that $\chi_i(\rho)=1$ if $|\rho|<\frac{1}{16}$ and $\chi_i(\rho)=0$ if $|\rho|\geq\frac{1}{15}$. The vector field
\[
X:=\nabla_{\mathbb R^2} f
\]
then satisfies the contraction property stated in Proposition~\ref{prop:contrac} on each band $B_i$. By construction, the vector field $X$ is Turing complete in the sense of Definition~\ref{TC}, but details are provided in the next step (see Theorem~\ref{T:turingpert}).

\begin{Remark}\label{rem:comput}
{Notice that the gradient field $X$ is computable in the following sense. Taking a concrete realization of the functions $\Lambda_i(s)$, e.g. $\frac{K_i}{16C}e^{-Cs}$, it is clear that for any compact set $K$ of the plane the expression of $X$ only depends on the curves $\gamma_i$, which in turn only depend on finitely many steps of the Turing machine (for finitely many inputs). Therefore, there is an algorithm that approximates $X$ in $K$ with arbitrary precision. The computation, however, is not uniform in the sense that one needs more and more computations of the Turing machine as the compact set grows.}
\end{Remark}

\subsection{Step~3: The Turing complete vector field $\nabla_{\mathbb R^2}f$ is weakly robust}\label{SS.wr}

In this last step, we prove that the computational power of the vector field $X$ is weakly robust in the sense that it is preserved by better-than-uniform perturbations. To this end, we first show that any vector field $Y$ that is close enough to $X$ has a collection of integral curves that are very close to $\gamma_i$ in a precise sense:

\begin{prop}\label{thm:PerturbedFlow}
There exists a smooth error function $\epsilon(x,y)$ that tends to zero at infinity fast enough such that for any vector field $Y$ that satisfies the estimate~\eqref{eq_wrob}, its trajectory $(s(t),\rho(t))$ with initial condition $(s(0),\rho(0))=p^i_0$ satisfies that
\begin{equation*}
(s(t),\rho(t)) \in \widetilde B_i
\end{equation*}
and
\[
\frac{ds}{dt}>0
\]
for all $t\geq0$ and $i\geq0$. Here $\widetilde B_i \subset B_i$ is a band such that
$$\widetilde B_i \cap \Big\{|y-l|<\frac{\varepsilon}{2}\Big\} \subset I^i_{\Delta^l (c_i)} \times \Big(l-\frac{\varepsilon}{2},l+\frac{\varepsilon}{2}\Big)$$
for all $l\geq0$.
\end{prop}

\begin{proof}
Proceeding as in Section~\ref{SS.grad}, we endow each band $B_i\subset\mathbb R^2$ with coordinates $(s,\rho)\in \mathbb R\times (-1/16,1/16)$. For concreteness we take
\[
\Lambda_i(s)=\frac{K_i}{16C}e^{-Cs}
\]
for $s\geq -1$ in the definition of $X_i$. In the set $B_i$ we can write the vector field $Y$ as
\[
Y=X_i+P
\]
where $P$ is a perturbation that is bounded pointwise as $|P(x,y)|\leq \epsilon(x,y)$. In the $(s,\rho)$ coordinates we denote the integral curve $(x(t),y(t))$ of $Y$ as $(s(t),\rho(t))$, and we recall that the point $p^i_0$ corresponds to the initial conditions $s(0)=0$ and $\rho(0)=0$. This trajectory then satisfies the ODE
\begin{equation}\label{eq:perturbedflow}
\begin{cases}
\frac{ds}{dt}=\frac{\Lambda_i(s)}{(1-\kappa_i(s(t))\rho(t))}+P_s(s(t),\rho(t))\,,\\
\frac{d\rho}{dt}=-\rho(t) + P_\rho(s(t),\rho(t))\,,
\end{cases}
\end{equation}
where $P_s,P_\rho$ stand for the components of the vector field $P$ in the $(s,\rho)$ coordinates. It is obvious that if $\|P\|_{C^0(\mathbb R^2)}$ is small enough, the band $B_i$ is forward invariant by the flow of $Y$. Additionally, since $\kappa_i(s)\rho\geq \frac{-15}{16}$ on $B_i$, we obtain
$$\frac{ds}{dt} \geq \frac{16\Lambda_i(s)}{31} + P_s(s(t),\rho(t))\,.$$
We then deduce that
\begin{equation}\label{eq:ds}
\frac{ds}{dt} > \frac{8\Lambda_i(s)}{31}= \frac{K_i}{62C} e^{-Cs}=:\Gamma_i(s)
\end{equation}
for all $t\geq0$, provided that
\begin{equation}\label{eq:firstconstraint}
\text{inf}_{|\rho|<\frac{1}{16}}|P_s(s,\rho)|<\frac{{8}\Lambda_i(s)}{31}
\end{equation}
on each $B_i$ for all $s\in\mathbb R$. This actually holds for all forward trajectories of $Y$ with initial condition in $B_i$. Accordingly, the vector field $Y$ is transverse to the lines $\{s=\text{constant}\}$ in each band $B_i$.

Let us now compare the evolution of the trajectories of $X$ and $Y$ on a compact set $K$. For this, we take integral curves $\widetilde \xi(t)$ and $\xi(t)$ of the vector fields $Y$ and $X$, respectively, with the same initial condition in $K$. A standard computation shows that
\begin{equation*}
\frac{d}{dt}\Big|\widetilde\xi-\xi\Big|\leq \Big|X(\widetilde\xi)-X(\xi)\Big|+\epsilon_K\,,
\end{equation*}
where we have set
\[
\epsilon_K:=\|P\|_{C^0(K)}\,.
\]
Applying the mean value theorem with
$$M_K:=\|DX\|_{C^0(K)}\,,$$
we obtain
\[
\frac{d}{dt}\Big|\widetilde\xi-\xi\Big|\leq M_K\Big|\widetilde\xi-\xi\Big|+\epsilon_K\,,
\]
and hence we can conclude, using Gronwall's inequality, that
\begin{equation}\label{eq:gron}
\Big|\widetilde\xi(t)-\xi(t)\Big|\leq \frac{\epsilon_K}{M_K}\Big(e^{M_Kt}-1\Big)\,,
\end{equation}
for all $t\geq0$ such that $\widetilde\xi(t)$ and $\xi(t)$ are contained in $K$.

We are now ready to apply this estimate to the trajectory $(s(t),\rho(t))$ in the statement; in order to compare with the unperturbed field $X$, we use the notation $(S(t),R(t))$ for the corresponding integral curve of $X$ with the same initial condition as $(s(t),\rho(t))$. Let $T^i_1(c)$ be the time for which $s(T^i_1(c))=s^i_1+c$, with $|c|<\varepsilon/2$. Let us now compute an upper bound of $T^i_1(c)$. From Equation~\eqref{eq:ds} it follows that $s(t)> \frac{1}{C} \ln (\frac{K_it}{62}+e^{Cs(0)})$ and hence $t<\frac{62}{K_i}(e^{Cs(t)}-e^{Cs(0)})$. We then infer that
\[
T^i_1(c) < \frac{62}{K_i}(e^{C(s^i_1+c)}-1)< \frac{62}{K_i}(e^{3C}-1)<\frac{62000}{K_i}=:\tau^i_1\,,
\]
where we used the simple estimate $s^i_1+c<3$ and that $e^{3C}-1=e^{3\ln10}-1<1000$. Notice that the constant $\tau^i_1$ only depends on the band $B_i$.

Next we take a compact set given by the box
$$K^i_0:=[2i,2i+1]\times [-1/4,1+1/4]\,.$$
Accordingly, if we assume that
\[
\epsilon_{K^i_0}<\epsilon^i_{0}:=\frac{M_{K^i_0}}{e^{M_{K^i_0}\tau^i_1}-1}\,\text{min}\bigg\{\frac{\varepsilon}{2},\frac{A_{i,1}}{4}\bigg\}
\]
we conclude from Equation~\eqref{eq:gron} that
$$
|S(T^i_1(c))-s^i_1|<\varepsilon
$$
for all $|c|<\varepsilon/2$, and also
\[
|\rho(T^i_1(c))|<\frac{A_{i,1}}{4}\,,
\]
where we have used that $R(t)=0$ for all $t\geq0$ ($X$ is tangent to the curves $\gamma_i$). In fact, it also holds that
\[
\rho(t)<\frac{A_{i,1}}{4}<\frac{A_{i,0}}{4}
\]
for $t\in[0,T^i_1(\varepsilon)]$.
In summary, we have established that for all $i\geq0$, and $l=0,1$, the intersection of the curve $(x(t),y(t))$ with the set $\{|y-l|<\frac{\varepsilon}{2}\}$ is contained in the box
\begin{equation*}
\Big\{|x-(\Delta^l(c_i)+2i)|<\frac{A_{i,l}}{4}\Big\}\times \Big\{|y-l|<\frac{\varepsilon}{2}\Big\}\subset I^i_{\Delta^l(c_i)}\times \Big(l-\frac{\varepsilon}{2},l+\frac{\varepsilon}{2}\Big)\,,
\end{equation*}
where the inclusion follows from the estimate in Lemma~\ref{lem:intsize}.

Now we claim that for each $i\geq0$ and $l\geq0$ the intersection of the curve $(x(t),y(t))$ with the set $\{|y-l|<\frac{\varepsilon}{2}\}$ is contained in the box
\begin{equation*}
\Big\{|x-(\Delta^{l}(c_i)+2i)|<\frac{A_{i,l}}{4}\Big\}\times \Big\{|y-l|<\frac{\varepsilon}{2}\Big\}\subset I^i_{\Delta^{l}(c_i)}\times \Big(l-\frac{\varepsilon}{2},l+\frac{\varepsilon}{2}\Big)\,,
\end{equation*}
for a suitable choice of the error function $\epsilon(x,y)$.

Since we have proved this result for $l\in\{0,1\}$, it is clear that the claim follows if we establish the following induction property: consider an integral curve of $Y$, which we still denote by $(s(t),\rho(t))$, with initial condition
\begin{equation*}
s(0)=s^i_l\,, \qquad |\rho(0)|<\frac{A_{i,l}}{4}
\end{equation*}
for some $i\geq0$ and $l\geq1$, then, if $T^i_{l+1}(c)$ denotes the time such that $s(T^i_{l+1}(c))=s^i_{l+1}+c$, with $|c|<\varepsilon/2$, we have the estimate
\begin{equation} \label{eq:mainestimate}
|\rho(T^i_{l+1}(c))|<\frac{A_{i,l+1}}{4}\,.
\end{equation}

To check this, we denote, as before, $(S(t),R(t))$ the trajectory of the vector field $X$ with the same initial condition $(s(0),\rho(0))$; this will allow us to exploit the contraction properties of the flow of $X$. As in the case $l=0$, it is clear that $T^i_{l+1}(c)$ is bounded by
\[
T^i_{l+1}(c)<\frac{62}{K_i}(e^{C(s^i_{l+1}+c)}-e^{Cs^i_l})< \frac{62}{K_i}e^{C(2l+3)}=:\tau^i_{l+1}\,,
\]
where we used that $|s^i_l|<2l$ and $|c|<1$, a bound that depends both on $i$ and $l$.

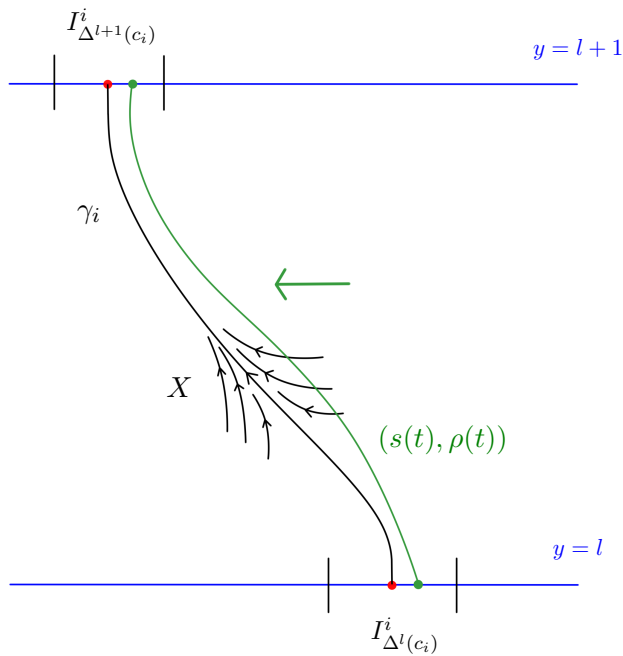
\begin{figure}[!h]
\begin{center}
\begin{tikzpicture}
     \node[anchor=south west,inner sep=0] at (0,0) {\includegraphics[scale=0.15]{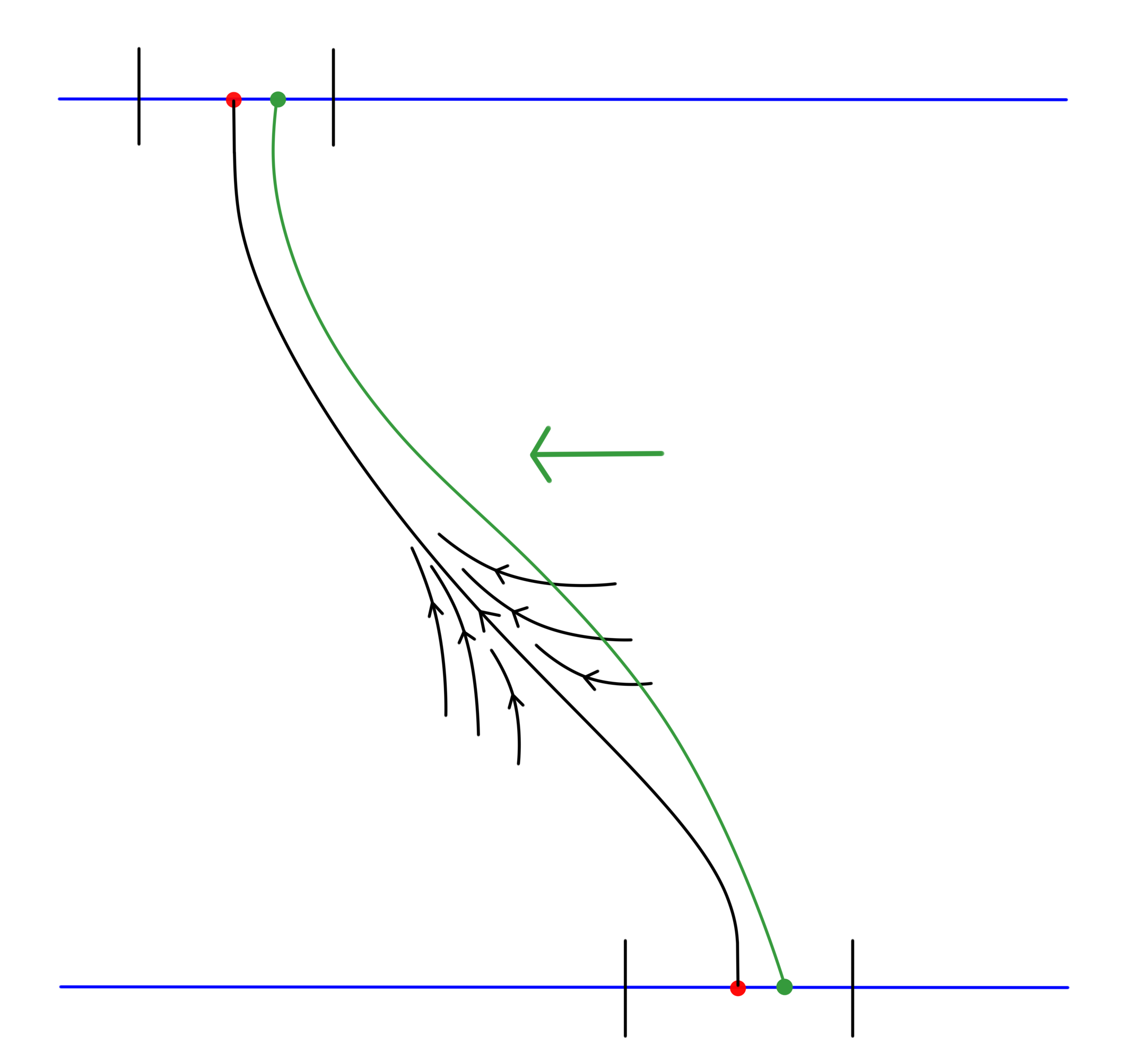}};

     \node[color=black!50!green] at (6.2,2.5) {$(s(t),\rho(t))$};

     \node at (1.5, 5.5) {${\gamma_i}$};

     \node at (2.7,3.2) {$X$};

     \node [scale=0.9] at (5.7, -0.1) {$I^i_{\Delta^l(c_i)}$};

     \node [scale=0.9] at (1.8, 8) {$I^i_{\Delta^{l+1}(c_i)}$};

     \node [blue, scale=0.8] at (8, 7.7) {$y=l+1$};
	\node [blue, scale=0.8] at (8, 1.05) {$y=l$};
\end{tikzpicture}
\caption{The contraction in the normal direction controls the perturbed flow-line}
\label{fig:contraction}
\end{center}
\end{figure}

To apply the estimate~\eqref{eq:gron}, we take the compact set
$$K^i_l:=[2i,2i+1]\times [l-1/4,l+1+1/4]\,,$$
and the small constant
\begin{equation*}
\epsilon_{K^i_l}:=\|P\|_{C^0(K^i_l)}<\epsilon^i_{l}:=\frac{M_{K^i_l}}{e^{M_{K^i_l}\tau^i_{l+1}}-1}\,\text{min}\bigg\{\frac{\varepsilon}{2},\frac{A_{i,l+1}}{8}\bigg\}\,.
\end{equation*}

It then follows from~\eqref{eq:gron} that $S(T^i_{l+1}(c))$ is bounded as
\[
|S(T^i_{l+1}(c))-s^i_{l+1}|<\varepsilon
\]
for all $|c|<\varepsilon/2$. Therefore, Proposition~\ref{prop:contrac} yields that
\[
|R(T^i_{l+1}(c))|<\frac{A_{i,l+1}}{8}\,,
\]
which implies in turn, applying again Equation~\eqref{eq:gron}, the desired bound
\[
|\rho(T^i_{l+1}(c))|<\frac{A_{i,l+1}}{4}\,,
\]
for all $|c|<\varepsilon/2$.

In summary, we have proved that if we take an error function $\epsilon(x,y):=|P(x,y)|$ satisfying the condition~\eqref{eq:firstconstraint} and such that
\[
\|\epsilon\|_{C^0(K^i_l)}<\epsilon_0\epsilon^i_{_l}
\]
for all $l\geq0$ and $i\geq0$, for some small enough constant $\epsilon_0>0$, then the integral curve $(s(t),\rho(t))$ of the perturbed field $Y$ is contained in a band $\widetilde B_i\subset B_i$ whose intersection with the set $\{|y-l|<\varepsilon/2\}$ is contained in the box
$$I^i_{\Delta^l(c_i)}\times (l-\varepsilon/2,l+\varepsilon/2)\,.$$
It is evident that such a smooth function $\epsilon(x,y)$ exists, which completes the proof of the proposition.
\end{proof}

\begin{Remark}\label{R:error2}
A simple analysis of the vector fields $X_i$ and a suitable extension to a global field $X$, allow us to conclude that the constant $$M_{K^i_l}:=\|DX\|_{C^0(K^i_l)}$$
can be taken to be uniform in $i$ and $l$. Since the asymptotic behavior of the constants $\tau^i_l$, $K_i$ and $A^i_l$ is
\begin{align*}
\tau^i_l\approx c_1e^{c_2(i+l)}\,,\qquad K_i\approx c_3e^{-c_4i}\,, \qquad A^i_l\approx c_5e^{-e^{c_6(i+l)}}
\end{align*}
for some positive constants $c_j$ (independent of $i$ and $l$), it is not hard to obtain a bound for the decay of the error function $\epsilon(x,y)$, which turns out to be a double exponential:
\begin{equation*}
|\epsilon(x,y)|<ce^{-e^{cr}}
\end{equation*}
for all $(x,y)\in\mathbb R^2$, with $r:=(x^2+y^2)^{1/2}$ and some $c>0$.
\end{Remark}

We can now prove that any suitably small perturbation of $X$ is Turing complete, which implies the main result of this section, cf. Theorem~\ref{th.turingrob}.

\begin{theorem}\label{T:turingpert}
Let $\epsilon\in C^\infty(\mathbb R^2)$ be the error function obtained in Proposition~\ref{thm:PerturbedFlow}. Then any smooth vector field $Y$ that is bounded pointwise as
$$ |X(x,y)-Y(x,y)| \leq \epsilon(x,y)\,,$$
is Turing complete.
\end{theorem}

\begin{proof}
Let $T$ be a universal Turing machine and $(q_0,t_{in})$, $t_{in}\in A$, an input configuration. It is clear that there is an algorithmic procedure (requiring a finite number of steps) that determines the position of the input $(q_0,t_{in})$ with respect to the order in $C_0$ {introduced in Section \ref{SS.curves}} and its value $c_i\in[0,1]$. Hence the input configurations correspond to the {explicitly computable} points $p^i_0=(c_i+2i,0)\in\mathbb R^2$.

Given an output string $t^*=(t^*_{-k},...,t^*_k) \in \Sigma^{2k+1}$, let $J^i_{t^*}$ denote the union of all the intervals $I_{(q,t)}^i$ for each configuration $(q,t)$ such that $q=q_{halt}$ and the symbols in position $-k,...,k$ of $t$ coincide with $t^*_{-k},...,t_k^*$. It is convenient to define the open set
\begin{equation}\label{eq:TCsets}
 U_{t^*}^{i}= \bigcup_{l=1}^\infty J^i_{t^*} \times (l-\varepsilon/2,l+\varepsilon/2) \subset [2i,2i+1]\times \mathbb{R}\,
\end{equation}
where $\varepsilon>0$ is the small constant that was fixed before. Notice that this is consistent with the assignment introduced in Definition~\ref{D:ep}. As usual, we say that an output $t\in A$ coincides with $t^*$ if the positions $-k,...,k$ of $t$ coincide with $t^*_{-k},...,t_k^*$.

To prove that the vector field $Y$ is Turing complete, cf. Definition~\ref{TC}, first assume that $T$ halts with input $t_{in}$ after $k$ steps with configuration $\Delta^k(c_i)=(q_{halt},t)$ and $t$ coincides with $t^*$. Let $(x(t),y(t))$ be the integral curve of $Y$ with initial condition $p^i_0$. By Proposition~\ref{thm:PerturbedFlow}, if $t_k$ denotes the time such that $y(t_k)=k$, it follows that
$$ x(t_k) \in I_{\Delta^k(c_i)}^i\,. $$
But $\Delta^k(c_i)$ is a halting configuration whose tape coincides with $t^*$, hence
$$I_{\Delta^k(c_i)}^i \subset J^i_{t^*}\,,$$
and therefore the trajectory intersects the open set $U_{t^*}^{i}$.

Conversely, assume that the trajectory of $Y$ with initial condition $p^i_0$ intersects the open set $U_{t^*}^{i}$. Let $l>0$ be the smallest integer such that the orbit intersects $J^i_{t^*}\times (l-\varepsilon/2,l+\varepsilon/2)$. By the definition of $J^i_{t^*}$, there is a configuration $(q_{halt},t)$, $t$ coinciding with $t^*$, such that the orbit intersects the box $I^i_{(q_{halt},t)} \times (l-\varepsilon/2,l+\varepsilon/2)$. By Proposition~\ref{thm:PerturbedFlow}, and the fact that all the intervals $I^i_{(q,t)}$ are disjoint, it follows that
$$I^i_{(q_{halt},t)}=I^i_{\Delta^l(c_i)}\,.$$
We then conclude that $\Delta^l(c_i)=(q_{halt},t)$, proving that $T$ halts with an output $t$ that coincides with $t^*$ after $l$ steps of the algorithm.
\end{proof}

\begin{Remark}
As mentioned in Section~\ref{SS:basic}, a consequence of the Turing completeness of $Y$ is that it exhibits undecidable trajectories. The open sets for which the reachability problem is undecidable are of the form $U^i_{t^*}$. It is easy to construct even simpler open sets for which the reachability problem is undecidable. Following Remark \ref{rem:haltdist}, we could have encoded every halting configuration in intervals that are at a positive distance of the non-halting configurations. In such a case, it is easy to check that the machine with input $c_i$ will halt if and only if the trajectory of the vector field with initial point $p^i_0$ intersects the half-plane $U^i_{halt}:=\{x> 2i+1-\frac{1}{2^{q_{halt}}} \}$.
\end{Remark}


\section{Proof of Theorem~\ref{th.main}}\label{S:main}

Let $f\in C^\infty(\mathbb R^2)$ be the function constructed in Theorem~\ref{th.turingrob} and $\epsilon(x,y)$ the error function that bounds the weak robustness of $\nabla_{\mathbb R^2}f$. Applying the better-than-uniform approximation theorem in~\cite{FG} we infer that there exists an entire function $F$ on $\mathbb R^2$ such that
\[
\text{max}_{|\alpha|\leq 1}\Big|D^\alpha(f(x,y)-F(x,y))\Big|\leq \epsilon(x,y)
\]
for all $(x,y)\in\mathbb R^2$. By Theorem~\ref{th.turingrob}, we then deduce that the vector field $\nabla_{\mathbb R^2} F$ is Turing complete as well.

Next, we invoke the Cauchy-Kovalevskaya Theorem~\ref{thm:main1} to conclude that, for any constant $\lambda\neq0$, there exists a Beltrami field $u$ on $\mathbb R^3$ satisfying $\curl u=\lambda u$ such that
\[
u(x,y,0)=\frac{\partial F}{\partial x}\partial_x+\frac{\partial F}{\partial y}\partial_y\,.
\]
It is clear that $u$ inherits the Turing completeness of the gradient field $\nabla_{\mathbb R^2} F$. Specifically, using the notation introduced in the proof of Theorem~\ref{T:turingpert}, the points and open sets associated with the initial configurations and halting states in Definition~\ref{TC} are given by
\begin{align}\label{eqpoints}
(c_i+2i,0,0)\in\mathbb R^3\\
U^i_{t^*}\times \Big(-\frac{\varepsilon}{2},\frac{\varepsilon}{2}\Big)\subset\mathbb R^3 \label{eq:sets}
\end{align}
for all $i\geq0$. Since the plane $\{z=0\}$ is invariant by the flow of $u$, it is straightforward to check that it is Turing complete. The theorem then follows.

\begin{Remark}\label{rem:comput2}
{Although the Beltrami field $u$ is obtained from $\nabla_{\mathbb R^2}F$ using the Cauchy Kovalevskaya theorem, the trajectories associated with the computations of the universal Turing machine simulated by $u$ only depend on $F$, and can be computed up to an arbitrarily small error according to Remark~\ref{rem:comput}. The tape size and the number of steps of this computation depend on the values $|x(t)|$ and $|y(t)|$, respectively, of the integral curves of $\nabla_{\mathbb R^2}F$.}
\end{Remark}

As a consequence of the proof we infer the existence of an analytic function on $\mathbb R^2$ (in fact, entire in the sense that it can be extended to a holomorphic function on $\mathbb C^2$) whose gradient is Turing complete:

\begin{Corollary}\label{Cor:TCan}
There exists an entire function $F:\mathbb R^2\to\mathbb R$ such that its gradient $\nabla_{\mathbb R^2} F $ is Turing complete.
\end{Corollary}

\begin{Remark}\label{rem:settoset}
Observe that the step-by-step evolution of a universal Turing machine is embedded into the planar gradient field $\nabla_{\mathbb R^2} F$. An input $c_i$ is associated with the point $p^i_0$, and the evolution of the flow intersects the open sets introduced in Equation~\eqref{eq:encod} following the computational process of the machine. However, we could have taken any initial point in the neighborhood
$$\frac14I^i_{c_i}\times \Big(\frac{-\epsilon}{2},\frac{\epsilon}{2}\Big)$$
of $p^i_0$, where the interval $I^i_{c_i}$ was introduced in Section~\ref{SS.curves}. Indeed, the evolution of the trajectory through a point in that neighborhood will also intersect the open sets~\eqref{eq:encod} following the steps of the machine with input $c_i$. This follows from the proof of the weak robustness of the computational power of the field, since the estimate~\eqref{eq:mainestimate} also holds for $l=0$. In particular, the Turing completeness of this gradient field can be stated in terms of a ``set-to-set'' property (see~\cite{DKB}), where instead of assigning a point to an input we accept any point in a neighborhood of $p^i_0$. In our construction, however, the interval $I^i_{c_i}$ is not uniform, and it becomes arbitrarily small for configurations whose tape $t$ has large support.
\end{Remark}

\section{Robust computational complexity}\label{S:bounded}

It is well known that Euclidean Beltrami fields do not have finite energy~\cite{Na14}. In fact, the Turing complete Beltrami field $u$ constructed in Theorem~\ref{th.main} is not granted to be even bounded. For this reason, it is customary to consider Beltrami fields on the $3$-torus endowed with the canonical flat metric. In this section, we explore the computational properties of Beltrami fields on the $3$-torus, showing that they can simulate an arbitrary tape-bounded Turing machine. We also provide a quantitative estimate relating the size of the tape with the energy of the field and compare it with the space-bounded Church-Turing thesis in the theory of computation~\cite{BSR}. These considerations, together with the recently developed theory of Gaussian random Beltrami fields~\cite{EPR}, allow us to prove that, with probability~$1$, a random Beltrami field in $\mathbb R^3$ is Turing complete (in a weaker sense to be specified later), thus establishing that the computational complexity of a Beltrami field is a typical phenomenon.

\subsection{Robust simulation of tape-bounded Turing machines}\label{SS:torus}

A Beltrami field on the flat torus $\mathbb T^3:=(\mathbb R/2\pi\mathbb Z)^3$
is an eigenfield of the curl operator, i.e.,
\[
\curl v= \lambda v
\]
for some eigenvalue $\lambda$. The spectrum of
the curl operator on the $3$-torus consists of the numbers of the form
$\lambda=\pm|k|$ for some vector with integer coefficients $k\in\mathbb
Z^3$.

The proof of Theorem~\ref{th.main} makes crucial use of the non-compactness of $\mathbb R^3$. Indeed, to apply the Cauchy-Kovalevskaya Theorem~\ref{thm:main1} we need to construct an entire function on $\mathbb R^2$ whose gradient is Turing complete. The non-compactness of $\mathbb R^2$ makes possible the construction of a smooth gradient field whose computational power is weakly robust under better-than-uniform approximations, cf. Theorem~\ref{th.turingrob}, which allows us to obtain an entire field by perturbation. This perturbative strategy breaks down when one considers smooth Turing complete vector fields on a compact surface: while these fields exist, see Section~\ref{S:final}, their Turing completeness is lost when  approximated  by analytic fields.\\

At this moment we do not know how to construct a Turing complete Beltrami field on $\mathbb T^3$. Instead, we can simulate a weaker notion of computability that consists in Turing machines whose tape size is bounded. We take here a concrete definition of bounded machines that is suited to our construction. This definition slightly differs from the usual definition of space-bounded Turing machines\footnote{{Here by a space-bounded Turing machine we mean a machine with a finite number of possible configurations.}}, but they enjoy the same computational power. A \emph{tape-bounded Turing machine} $T_b$ is given by a usual Turing machine $T$ together with an allowed size of the tape, i.e., two integer numbers $b_-$ and $b_+$, with $b_-\leq 0<b_+$, such that the only positions of the tape that are allowed are $b_{-},...,b_+$. In particular, the space of configurations is of the form $Q\times \Sigma^{N}$ for some natural number $N$. For a given input, the machine $T_b$ either halts to give an output (without leaving the allowed size of the tape), enters a periodic orbit of configurations or attempts to reach a position of the tape that is not allowed (hence we say that it ``runs out of memory").

To simulate tape-bounded Turing machines in a compact space, we use a deep property of Beltrami fields on $\mathbb T^3$ first shown in~\cite{EPT} and known as ``the inverse localization property''. Roughly speaking, it says that any Beltrami field on a compact set of $\mathbb R^3$ can be reproduced in a ball of size $\lambda^{-1}$ by a Beltrami field on $\mathbb T^3$ with high enough eigenvalue $\lambda$. For the precise statement let us fix an arbitrary point $p_0\in \mathbb T^3$ and take a patch of normal geodesic coordinates $\Psi:\mathbb B\to B$ centered at $p_0$. Here $B$ (resp.~$\mathbb B$) denotes the unit
ball in~$\mathbb R^3$ (resp.~the geodesic unit ball in~$\mathbb T^3$) centered at the origin (resp.\ at $p_0$). The following result was proved in~\cite{EPT}:
\begin{theorem}\label{th-il}
Let $u$ be a Beltrami field in $\mathbb R^3$, satisfying
$\curl u=u$, $B_N\subset \mathbb R^3$ the ball centered at the origin of radius $N$, and fix any positive real $\delta$ and integer $m$. There is a positive integer $L_0$ such that for any odd integer $L>L_0$ there exists a
Beltrami field~$v_L$ on $\mathbb T^3$, satisfying $\curl v_L=L v_L$, such that
\begin{equation*}\label{nose2}
\bigg\|u-v_L\circ \Psi^{-1}\bigg(\frac{\cdot}{L}\bigg)\bigg\|_{C^m(B_N)}<\delta\,.
\end{equation*}
\end{theorem}


We are now ready to prove that any tape-bounded Turing machine can be simulated by a Beltrami field on $\mathbb T^3$ in the following sense: a vector field $v$ on $\mathbb T^3$ simulates the tape-bounded Turing machine $T_b$ if there exists a compact set $K\subset \mathbb T^3$ such that for any integer $k\leq \frac{s(T_b)-1}{2}$ ($s(T_b)$ is the tape size of $T_b$), an input tape $t$, and a finite string $(t_{-k}^*,...,t_k^*)$ of symbols of the alphabet, there exist an explicitly constructible point $p\in K$ and an open set $U\subset K$ such that the orbit of $v|_K$ through $p$ intersects $U$ if and only if $T_b$ halts (without running out of memory) with an output tape whose positions $-k,...,k$ correspond to the symbols $t_{-k}^*,...,t_k^*$. This definition is weaker than Definition~\ref{TC} in the sense that we only consider the piece of the integral curve of $v$ starting at $p$ that is contained in $K$; with the usual concept of simulation we can safely take $K=\mathbb T^3$, but we cannot do this here because we lose the control of the dynamics of $v$ outside $K$ (the integral curves could eventually come back to $K$). Notice that a tape-bounded Turing machine may also halt if it runs out of memory, but this is a ``fake'' halting, so it is not associated with the intersection of an open set with the orbits of the vector field that simulates the machine.

\begin{theorem}\label{coro:bounded}
Let $T_b$ be a tape-bounded Turing machine. Then there exists a Beltrami field $v$ on $\mathbb T^3$ that robustly simulates $T_b$.
\end{theorem}
\begin{Remark}
The meaning of robust simulation is that any other vector field on $\mathbb T^3$ that is close enough to $v$ in the $C^0$-norm also simulates the tape-bounded Turing machine $T_b$. The robustness also holds with respect to initial conditions by the same argument, using the continuous dependence of the flow on initial conditions. In particular, as explained in Remark~\ref{rem:settoset}, the simulations can be understood as a ``set-to-set" property~\cite{DKB}. Since only a finite number of configurations of the universal Turing machine are taken into account, the size of the open sets described in~\eqref{eq:encod} encoding the configurations is uniformly bounded from below.
\end{Remark}
\begin{proof}
Let us consider the Turing complete Beltrami field $u$ in $\mathbb R^3$ obtained in Theorem~\ref{thm:main1}; by construction, it simulates a universal Turing machine $T$. If the tape size of $T_b$ is $s_b\equiv s(T_b)$, it is clear that there exists an integer $s\geq s_b$ such that $T$ simulates $T_b$ with tapes of size at most $s$. We can then take an integer $N\equiv N(s)$ such that the vector field $u$ restricted to the ball $B_N$ simulates $T_b$. More precisely, this means that if we consider the integral curves of $u$ starting at the points associated with the inputs of $T_b$ (via the correspondence established by $T$), one and only one of the following cases occurs:
\begin{enumerate}
\item The trajectory intersects a set $U_{(q,t)}$ (cf. Equations~\eqref{eq:encod} and~\eqref{eq:sets}) for some configuration $(q,t)$ with a tape whose tape size is greater than $s$. If this happens and no set of the form $U_{(q_{halt},t')}$ was reached before, the machine $T_b$ does not halt with that input (it runs out of memory).
\item The trajectory intersects a set $U_{(q_{halt},t)}$, and we are not in case~(1). Then $T_b$ halts with that input.
\item The trajectory leaves the ball $B_N$ and cases~(1) and~(2) do not occur. Then the machine $T_b$ does not halt with that input (it falls into a loop of non-halting configurations).
\end{enumerate}
We claim that the simulation of $T_b$ by $u$ is uniformly robust in the sense that there is $\epsilon\equiv \epsilon(s)>0$ such that any smooth vector field $\widetilde u$ that is close to $u$ as
\[
\|u-\tilde u\|_{C^0(B_N)}<\epsilon
\]
also simulates $T_b$. Indeed, for any $p\in B_N$, the continuous dependence of solutions to an ODE with respect to parameters implies that the distance between the integral curves $\gamma_p$ and $\widetilde \gamma_p$ of $u$ and $\tilde u$, respectively, is bounded as
\[
\text{dist}(\gamma_p\cap B_N,\widetilde\gamma_p\cap B_N)<C_N\epsilon\,,
\]
for some $\epsilon$-independent positive constant $C_N$. Then, taking $\epsilon$ small enough, it is clear from the construction of $u$ and the univocal association between configurations $(q,t)$ and disjoint open sets $U_{(q,t)}$, that the trajectories of $\widetilde u|_{B_N}$ starting at the points associated with the initial configurations of $T_b$ will intersect the halting open sets if and only if $T_b$ halts with corresponding output.

Finally, if we take a small constant $\delta<\frac{\epsilon}{2}$, the inverse localization Theorem~\ref{th-il} implies that there exists a Beltrami field $v$ on $\mathbb T^3$ of large enough eigenvalue $L>L_0(s)$ whose localization is close to $u$, cf. Equation~\eqref{nose2}. We conclude from the uniform robustness of $u$ that $v$ simulates the tape-bounded Turing machine $T_b$ in the compact set
$$K:=\Psi^{-1}\Big(\frac{B_N}{L}\Big)\,.$$
\end{proof}
{As in Remark \ref{rem:comput2}, the trajectories simulating the bounded Turing machine can be computed up to some negligible error which does not destroy the computational power of the simulation.}
\begin{Remark}
The diameter of the set $K\subset\mathbb T^3$ where $v$ simulates the tape-bounded Turing machine $T_b$ is of order
\[
\frac{N}{L}\ll 1\,.
\]
 When the tape size $s_b\to\infty$, also $N\to \infty$ and $v$ simulates greater in size tape-bounded Turing machines. However, we cannot take the limit to obtain a Turing complete Beltrami field on $\mathbb T^3$ because the eigenvalue satisfies $L\to \infty$, and hence the $H^1$-norm of $v$ is not uniformly bounded:
\[
\|v\|_{H^1(\mathbb T^3)}\to \infty\,.
\]
Quantitative estimates will be provided in the following subsection.
\end{Remark}

To conclude this subsection we want to emphasize that, although Theorem~\ref{coro:bounded} does not yield the (robust) presence of undecidable trajectories, it proves that the reachability problem can be of (robust) arbitrarily high computational complexity. {Concretely, for any given number $n$ of required computations, there exist a Beltrami field on $\mathbb T^3$ with a trajectory for which solving the reachability problem requires at least $n$ computations.}

\begin{Remark}\label{rem:NS}
As detailed in~\cite[Section 6B]{CMPP2}, given a Beltrami field $v$ on $\mathbb T^3$ with eigenvalue $\lambda>0$, we can take any rescaling $Mv$ (with a constant $M>0$) as the initial datum in the Navier-Stokes equations on $\mathbb T^3$ with viscosity $\nu>0$. The solution is then of the form
$$V(\cdot,t)=Mv(\cdot)e^{-\nu \lambda^2 t}$$
for some pressure function (whose explicit expression we omit).
The fluid particle paths then solve the non-autonomous ODE
$$\frac{dx(t)}{dt}=M e^{-\nu \lambda^2 t}v(x(t))\,,$$
which readily implies that the solution $x(t)$ for $t\in [0,\infty)$ travels the orbit of $v$ for times in the interval $[0,\frac{M}{\nu\lambda^2})$. Now, let $v_b$ be a Beltrami field as in Theorem~\ref{coro:bounded} that simulates a tape-bounded Turing machine $T_b$. Obviously, all the computations of $T_b$ simulated by the orbits of $v_b$ finish in finite time smaller than certain $T_{max}$. Hence, taking $M\equiv M_b$ such that $\frac{M}{\nu\lambda^2}>T_{max}$, the solution $V_b$ to the Navier-Stokes equations with initial datum $Mv_b$ also simulates $T_b$ robustly. The constant $M$ depends on the tape size $s_b$ of $T_b$, and it is clear that $M\to\infty$ as $s_b\to\infty$. Accordingly, the fluid particle paths of these time-dependent solutions to the Navier-Stokes equations on $\mathbb T^3$ also exhibit (robust) arbitrarily high computational complexity, which makes the reachability problem non-computable, from a practical point of view, in general.
\end{Remark}

In view of the classical result in the theory of computation (see e.g.~\cite{BGH}) that there cannot exist robust Turing complete systems on compact spaces, and only bounded computation is physically feasible, our result is optimal in the sense that we show that any tape-bounded computation can be robustly simulated by a Beltrami flow on $\mathbb T^3$.

\begin{Remark}
Another appealing way to interpret Theorem~\ref{coro:bounded} is using the language of cellular automata. As proved by Von Neumann in the 1940s, every Turing machine has a cellular automaton that simulates it. Turing machines with bounded tape are equivalent in computational power to finite-state automata, which can be simulated by a Beltrami field on $\mathbb T^3$ in view of our theorem.
\end{Remark}

\subsection{Quantitative estimates and robust simulation in physical systems}

In this section, we estimate the robustness $\epsilon\equiv\epsilon(s_b)$ and the $H^1$-norm of the Beltrami field $v$ on $\mathbb T^3$ that simulates a tape-bounded Turing machine of tape size $s_b$, which we construct in Theorem~\ref{coro:bounded}. In our view, this limitation actually provides additional support for the widely accepted conjecture in the theory of computation that physical memory, as a measure of robustness, limits the computational power of the system. The qualitative version of this conjecture has been proved, under certain assumptions, in~\cite{BGH}. A quantitative formulation has been recently introduced in~\cite{BSR}:
\\

\noindent\emph{Space-bounded Church-Turing thesis~\cite{BSR}:} A physical dynamical system $\mathcal{S}$ with memory $M$ is only capable of performing computations of tape-bounded Turing machines whose tape size is of order $M^{O(1)}$. The physical memory $M$ is a measure of the computational robustness of the system (i.e., the size $\epsilon$ of the perturbation that is allowed for the system to retain its computational power). There is no general formula that relates $M$ and $\epsilon$, but heuristic considerations and some rigorous results strongly suggest the following estimate~\cite{BRS}:
\[
M\sim \log \Big(\frac{1}{\epsilon}\Big)\,.
\]
\\

Theorem~\ref{coro:bounded} can be understood as further evidence towards this thesis: tape-bounded Turing machines can be robustly simulated by steady Euler flows (Beltrami fields) on $\mathbb T^3$, but our construction breaks down when the tape size tends to infinity. The following result bounds the robustness and the energy of the system in terms of the tape size of the Turing machine that it simulates:

\begin{prop}\label{L:robustestimate}
Let $v$ be the Beltrami field on $\mathbb T^3$ that simulates the tape-bounded Turing machine $T_b$ constructed in Theorem~\ref{coro:bounded}. If the tape size of $T_b$ is $s_b$, then the computational robustness of $v$ is bounded as
\begin{equation*}
0<\epsilon\leq c\exp(-\exp(\exp(c s_b)))\,,
\end{equation*}
and its $H^1$-norm is at least
\[
\|v\|_{H^1(\mathbb T^3)}\geq c\exp(\exp(\exp(cs_b)))
\]
for some positive constant $c$.
\end{prop}
\begin{Remark}
The space-bounded Church-Turing thesis is satisfied because the tape size $s_b$ in terms of the memory $M\sim \log(\epsilon^{-1})$ of the system is estimated as
\[
s_b\sim \log\log(M)\,.
\]
In view of this bound, we conclude that the Beltrami field $v$ that simulates $T_b$ is far from being efficient from the computational viewpoint.
\end{Remark}
\begin{proof}
We use the same notation as in the proof of Theorem~\ref{coro:bounded}. To simulate the tape-bounded Turing machine $T_b$, we need to take tapes of size $s\geq s_b$ in the computations of the universal Turing machine $T$ simulated by the Beltrami field $u$ in $\mathbb R^3$. Then $T$ computes $ce^{cs}$ steps for $ce^{cs}$ possible inputs of the machine (the constant $c>0$ only depends on the description of $T$, but not on $s$). Taking into account the way the Turing machine $T$ is encoded in the trajectories of $u$, cf. Section~\ref{S:TCgrad}, if we want to simulate the aforementioned number of steps and inputs, $l$ and $i$ in the notation of Section~\ref{SS.curves}, we need to take a ball of radius
\[
N\sim l+i \sim ce^{cs}\,.
\]
Remark~\ref{R:error2} then implies that the robustness $\epsilon$ of the computational ability of $u|_{B_N}$ to simulate $T_b$ is then at most
\[
\epsilon\leq c\exp(-\exp(\exp(c s)))\leq c\exp(-\exp(\exp(c s_b)))\,,
\]
as claimed.

To estimate the $H^1$-norm of $v$ it is enough to notice that the dependence between the eigenvalue $L$ and the error $\delta$ when doing the inverse localization is (see~\cite{EPT})
\[
L\geq \frac{c}{\delta}\,,
\]
so taking into account that we consider $\delta<\epsilon/2$, the previous estimates finally yield
\[
\|v\|_{H^1(\mathbb T^3)}\sim (1+L)^{1/2} \geq c\exp(\exp(\exp(c s_b)))\,,
\]
where we have normalized the $L^2$-norm of $v$, $\|v\|_{L^2(\mathbb T^3)}=1$. This completes the proof of the proposition.
\end{proof}

From the computational viewpoint, the proof of Proposition~\ref{L:robustestimate} shows that the complexity of the simulation is bounded by the eigenvalue of the Beltrami field, and hence by its $H^1$-norm $E:=\|v\|_{H^1(\mathbb T^3)}$. In other words, there is an increasing function $f(E)\sim C\log\log\log(E)$ of the energy $E$ of the field, which measures the size of the tape-bounded machine that is robustly simulated in terms of the energy of the field. This defines a (space) complexity class \textsc{SPACE}$(f(E))$ for the complexity of the reachability problem for steady Euler flows on $\mathbb T^3$.

\subsection{The computational complexity of Beltrami fields is typical}

A natural question is to what extent a random Beltrami field in Euclidean space is complex from the point of view of computability, either in terms of Turing completeness or of the computational complexity of its trajectories. Surprisingly enough this problem can be addressed using the previous considerations on tape-bounded Turing machines and the theory of Gaussian random Beltrami fields introduced in~\cite{EPR}. Inspired by the celebrated theory of Gaussian random monochromatic waves of Nazarov and Sodin, the authors of~\cite{EPR} constructed a Gaussian probability measure $\mu_B$ on the space of $C^k$ vector field on $\mathbb R^3$, where $k$ is any fixed nonnegative integer, whose support is the space of Beltrami fields. This measure has the following properties:

\begin{enumerate}
\item Let $u$ be a Beltrami field. For any compact set $K\subset\mathbb R^3$ and each $\epsilon>0$,
\[\mu_B(\{v\in C^k(\mathbb R^3,\mathbb R^3): \|u-v\|_{C^k(K)}<\epsilon\})>0\,.\]
\item The probability measure $\mu_B$ is translationally invariant and ergodic with respect to translations. The translation operator $\tau_y$, $y\in\mathbb R^3$, on $C^k$ vector fields is defined as $\tau_yw(x):=w(x+y)$.
\end{enumerate}

The main result of this subsection is that, using the Gaussian measure $\mu_B$, we can prove that, almost surely, a Gaussian random Beltrami field is Turing complete in the following sense:

\begin{theorem}\label{T:random}
With probability~$1$ a Gaussian random Beltrami field on $\mathbb R^3$ exhibits trajectories simulating {every possible} tape-bounded Turing machine. In other words, the field has trajectories of arbitrarily high computational complexity.
\end{theorem}
\begin{proof}
The first observation is that any tape-bounded Turing machine can be simulated by a countable set of tape-bounded Turing machines $T_{b_j}$ of tape size $s_j\to \infty$. Indeed, we can consider the universal Turing machine $T$ with an allowed tape of size $j$ in each direction, i.e., the allowed positions are $-j,...,j$. Let $u$ be the Turing complete Beltrami field constructed in Theorem~\ref{th.main}. We take a sequence of increasing positive numbers $\{N_j\}_{j=1}^\infty$, $N_j\to\infty$ as $j\to\infty$, and for each $N_j$ the corresponding ball $B_j\equiv B_{N_j}$ of radius $N_j$ centered at the origin. As explained in Section~\ref{SS:torus}, these constants can be taken so that the field $u|_{B_j}$ simulates the tape-bounded Turing machine $T_{b_j}$, and this computational ability is robust in the sense that there is a small enough constant $\epsilon_j$ such that any vector field $w$ that is close to $u$ as
\[
\|u-w\|_{C^0(B_j)}<\epsilon_j
\]
also simulates the same tape-bounded Turing machine $T_{b_j}$. Obviously, $s_j\to\infty$ and $\epsilon_j\to 0$ as $j\to\infty$ (because $u$ is a Turing complete vector field that is not robust under small uniform perturbations).

Now we can apply the theory of Gaussian random Beltrami fields. Property~(1) implies that for each $j$
\begin{equation}\label{muB}
\mu_B(\{v\in C^0(\mathbb R^3,\mathbb R^3): \|u-w\|_{C^0(B_j)}<\epsilon_j\})>0\,,
\end{equation}
i.e., the set of Beltrami fields that simulate the tape-bounded Turing machine $T_{b_j}$ has positive probability. Let us define the functional $\Phi_j$ on the space of vector fields $C^0(\mathbb R^3,\mathbb R^3)$ as $\Phi_j(w)=1$ if there is a compact set $K$ such that $w|_K$ simulates the tape-bounded Turing machine $T_{b_j}$, and $\Phi_j(w)=0$ otherwise. By construction, $\Phi_j$ is in $L^1$ with respect to the probability measure $\mu_B$, and the robustness of the computational ability of $w$ to simulate $T_{b_j}$ (which is clear because the simulated number of steps and inputs is finite) implies that $\Phi_j$ is lower semicontinuous, and hence measurable. We can then apply Property~(2) above (see also the ergodic theorem~\cite[Proposition~3.7]{EPR}) to conclude that, almost surely
\[
\lim_{R\to\infty}\Bint_{B_R} \Phi_j\circ\tau_y\, dy = \bE\Phi_j\,.
\]
Then, Equation~\eqref{muB} implies that $\bE\Phi_j>0$, and therefore for each Beltrami field $w$ on a full $\mu_B$-measure set, there is a ball of large enough radius $R$ such that $\Phi_j(w|_{B_R})=1$. This means that, with probability~$1$, a Gaussian random Beltrami field simulates the tape-bounded Turing machine $T_{b_j}$.

Finally, since the countable intersection of sets with probability~$1$ also has probability~$1$, we conclude that, almost surely, a Gaussian random Beltrami field simulates all the tape-bounded Turing machines $\{T_{b_j}\}_{j=1}^\infty$.
\end{proof}
{Notice that, in contrast with Theorem~\ref{thm:torus}, the randomly chosen Beltrami field in Theorem~\ref{T:random} can simulate every bounded-tape Turing machine. This ensures that, almost surely, a randomly chosen Beltrami field exhibits trajectories of arbitrarily high computational complexity.}
Even if every tape-bounded Turing machine is robustly simulated by a Gaussian random Beltrami field with probability one, the robustness of the simulation is not uniform: it depends on the size of the tape. This follows from the fact that $\epsilon_j$ tends to zero as $j$ goes to $\infty$.

\begin{Remark}
In a sense weaker than Definition~\ref{TC}, Theorem~\ref{T:random} establishes that a Gaussian random Beltrami field in $\mathbb R^3$ is Turing complete almost surely. Indeed, each orbit that halts of the universal Turing machine $T$ can be simulated by one of the tape-bounded Turing machines in the countable family $\{T_{b_j}\}_{j=1}^\infty$, and in this sense, we say that the family simulates $T$. The proof of Theorem~\ref{T:random} then shows that with probability~$1$ a random Beltrami field $w$ in $\mathbb R^3$ satisfies that, for each $j$, there is a domain $V_j\subset\mathbb R^3$ such that $w|_{V_j}$ simulates the Turing machine $T_{b_j}$; since this happens for all $j$, $w$ is universal. However, we need to consider a countable number of orbits instead of one (because for each input of the universal Turing machine $T$ we need to check the evolution of $T_{b_j}$ for all $j$), so it cannot be understood as Turing completeness in the usual sense. Another problem, although consistent with the fact that the Beltrami field is randomly chosen, is that the sets $V_j$ are not explicitly constructible as required in Definition~\ref{TC}.
\end{Remark}

\section{Turing complete dynamical systems on $\mathbb S^2$ with zero topological entropy}\label{S:final}

An interesting corollary of our construction in Section~\ref{S:TCgrad} is the existence of a Turing complete $C^\infty$ vector field on the $2$-sphere with zero topological entropy. Hence, we find an example of a flow that is complex from a computational viewpoint but not from a dynamical viewpoint (it is not chaotic).

\begin{prop}\label{P:entropy}
There exists a Turing complete $C^\infty$ vector field on $\mathbb S^2$ with zero topological entropy.
\end{prop}
\begin{proof}
Let $X$ be the weakly robust Turing complete vector field on $\mathbb{R}^2$ constructed in Theorem~\ref{T:turingpert}. It is obvious that there is a positive function $G\in C^\infty(\mathbb{R}^2)$, all whose derivatives tend to zero at infinity fast enough, so that
\[
\lim_{|x|\to\infty} |D^\alpha\widetilde X(x)|=0
\]
where $\widetilde X:=GX$ and $\alpha$ is any finite multi-index. Denoting by $\varphi: \mathbb{R}^2\longrightarrow \mathbb S^2\setminus\{N\}$ the inverse stereographic projection, the vector field
$$Y:=\varphi_*\widetilde X$$
extends smoothly to all $\mathbb S^2$ by adding a zero on the north pole $N$. Since the trajectories of the vector field $\widetilde X$ coincide with those of $X$, we infer that the vector field $Y$ on $\mathbb S^2$ is Turing complete as well. The points and open sets on $\mathbb S^2$ in the definition of Turing completeness are the images by $\varphi$ of the points and open sets making $X$ Turing complete on $\mathbb{R}^2$ (for details, see the last part of the proof of Theorem 4.1 in~\cite{CMP2}). Concerning the topological entropy of $Y$, we use the fact that any continuous (autonomous) flow on a compact surface has zero topological entropy~\cite{Y}.
\end{proof}
\begin{Remark}
In Corollary~\ref{Cor:TCan} we obtained a Turing complete analytic flow on $\mathbb{R}^2$. Unfortunately, the compactification procedure in the proof of Proposition~\ref{P:entropy} does not allow us to construct a Turing complete analytic flow on $\mathbb S^2$ because the construction does not preserve the analyticity of the field.
\end{Remark}

We will now argue as in~\cite[Theorem 1.3]{CMP2} to deduce that there exists a diffeomorphism of $\mathbb S^2$ that is Turing complete and has zero topological entropy. The definition of Turing completeness for diffeomorphisms is analogous to Definition~\ref{TC}, where the orbit is understood as the positive iterates of the diffeomorphism.

\begin{corollary}\label{cor:diffeo}
Let $\phi_t$ denote the flow of $Y\in \mathfrak{X}^\infty(\mathbb S^2)$. There exists some $\delta_0>0$ such that for every $\delta<\delta_0$, the map $\phi_{\delta}$ is a Turing complete diffeomorphism of $\mathbb S^2$ with zero topological entropy.
\end{corollary}

The key to prove this result is the following lemma. As before, $X$ is the Turing complete smooth vector field on $\mathbb R^2$ constructed in Theorem~\ref{T:turingpert}. In the statement, the input points $p^{i}\equiv p^i_0\in\mathbb R^2$, $i\geq 0$, were defined in Section~\ref{SS.curves}.

\begin{lemma}\label{lem:intersections}
Let $\chi_t$ denote the time-$t$ flow of $X$. There is a constant $\delta_0>0$ such that for every $\delta<\delta_0$, the following property holds. For every $p^i$, the orbit of $\chi_{\delta}$ through $p^i$ intersects the $\varepsilon/2$-neighborhood of each line $\{y=k\}$ for every $k\in \mathbb{N}$.
\end{lemma}
\begin{Remark}\label{rem:equiv}
We choose $\varepsilon/2$ instead of $\varepsilon$ because of the definition of the open sets $U_{t^*}^i$ in Equation~\eqref{eq:TCsets}. Denote by $U^i_{(q,t)}$ the open set  $\bigcup_{l=0}^\infty I_{(q,t)}^i \times (l-\varepsilon/2,l+\varepsilon/2)$. By the properties of $X$ (e.g. using that $X$ is tangent to the curves $\gamma_i$ and Lemma~\ref{L:curves}), the statement of the lemma is equivalent to the following property: for every input $c_i$ of the Turing machine with associated point $p^i$, the orbit of $\chi_{\delta}$ through $p^i$ satisfies that for any configuration $(q,t)$, there is some $r$ such that $\chi_{\delta}^r(p^i) \cap U^i_{(q,t)} \neq \emptyset $ if and only if there is some $k$ such that $\Delta^k(c_i)=(q,t)$.
\end{Remark}
\begin{proof}
By construction of the vector field $X$ (see Equation~\eqref{eq:fastestflow}), and the estimates of the constant $K_i$ and the function $\Lambda_i(s)$, there is some constant $c_0>0$ (that does not depend on the band $B_i$), such that the trajectory $(s(t),\rho(t))$ of $X$ through $p^i$ satisfies, in coordinates $(s,\rho)$ of the band $B_i$:
$$\frac{ds}{dt}<c_0\,.$$
Furthermore, by construction, the coordinate $y$ coincides with $s$ in the $\varepsilon$-neighborhood of each horizontal line $\{y=k\}$. This shows that the time that the solution takes to go from $y={k-\varepsilon/2}$ to $y=k+\varepsilon/2$ is at least $\tau=\frac{\varepsilon}{2c_0}$. Let us set
$$ \delta_0:= \frac{\varepsilon}{2c_0}\,.$$
It is then clear by the properties of the solution $(s(t),\rho(t))$, that for every $\delta<\delta_0$ and for all $k\geq 0$, there exists some $r\equiv r(\delta,k)$ such that $\chi_{\delta}^r(p^i)\cap (\mathbb R\times (k-\varepsilon/2,k+\varepsilon/2))\neq \emptyset$, thus completing the proof of the lemma.
\end{proof}

\begin{proof}[Proof of Corollary~\ref{cor:diffeo}]
Let $Y=\varphi_*\widetilde X$ be the smooth Turing complete vector field on $\mathbb S^2$ obtained in Proposition~\ref{P:entropy}. Recall that $\widetilde X=GX$ is just a reparametrization of the Turing complete vector field constructed in Section \ref{S:TCgrad}.

We can safely assume that $G\leq 1$ on $\mathbb R^2$, thus implying that $|\widetilde X|\leq |X|$ at any point. Hence if $\psi_t$ denotes the flow of $\widetilde X$, and $\delta_0$ is the constant in Lemma~\ref{lem:intersections}, for every $\delta<\delta_0$, the map $\psi_{\delta}$ also satisfies the conclusion in Lemma~\ref{lem:intersections}. By Remark~\ref{rem:equiv}, the orbit of $\psi_{\delta}$ through a point $p^i$ (associated with an input $c_i$) satisfies that for any configuration $(q,t)$, there is some $r$ such that $\psi_{\delta}^r(p^i) \cap U_{(q,t)}^i \neq \emptyset $ if and only if there is some $k$ such that $\Delta^k(c_i)=(q,t)$. This implies that the machine halts with input $c_i$ and some output $t^*=(t^*_{-k},...,t^*_{k})$ if and only if the orbit of $\psi_{\delta}$ intersects the open set $U^i_{t^*}$ (as defined in Equation~\eqref{eq:TCsets} in Section \ref{S:TCgrad}). By the construction of $Y$, its flow $\phi_t$ satisfies on $\mathbb S^2\setminus\{N\}$
$$\phi_t= \varphi \circ \psi_t \circ \varphi^{-1}. $$
This shows that the universal Turing machine halts with input $c_i$ if and only if the orbit of $\phi_{\delta}$ through $\varphi(p^i)$ intersects the open set $\varphi(U^i_{t^*})$. Accordingly, the $\delta$-time flow of $Y$ is Turing complete. Finally, the fact that the diffeomorphism $\phi_\delta$ has zero topological entropy follows from Abramov's formula. Indeed, if $h_{top}$ denotes the topological entropy of a diffeomorphism, we have
$$h_{top}(\phi_{\delta})=|\delta|h_{top}(\phi_1)=0\,,$$
where we used that the vector field $Y$ has zero topological entropy, which is defined as the topological entropy of the time-one flow.
\end{proof}

Observe that in~\cite{Mo1, CMPP2} Turing complete diffeomorphisms of the disk (that can be made area-preserving) are constructed, but no information on their topological entropy is given. In~\cite{GZ}, the authors construct closed-form analytic time-dependent ODEs in $\mathbb{R}^2$ that can robustly simulate a universal Turing machine. By projecting to the sphere as in~\cite{CMP2} while keeping track of the constructibility of the ODE, an explicit closed-form Turing complete smooth and time-dependent flow on $\mathbb S^2$ is obtained. Our construction has the advantage of being time-independent with zero topological entropy, at the cost of not having an explicit closed-form expression.

\end{document}